\documentclass[12pt,reqno]{amsart}

\usepackage{stmaryrd}
\usepackage{txfonts}
\usepackage{amssymb}
\usepackage{amsfonts}
\usepackage{}

\usepackage{amssymb,amsmath,graphicx,
	amsfonts,euscript}
\usepackage{color}
\usepackage{cite}
\allowdisplaybreaks

\newtheorem {theorem}{Theorem}[section]
\newtheorem {lemma}[theorem]{{\bf Lemma}}

\newtheorem {proposition}[theorem]{{\bf Proposition}}
\theoremstyle{remark}

\theoremstyle{plain} \numberwithin {equation}{section}

\def\ee{\epsilon}

\def\na{\nabla}
\def\th{\theta}
\newcommand{\p}{\partial}

\def\R{\mathbb{R}}
\def\N{\mathbb{N}}

\newcommand{\beq}{\begin{equation}}
\newcommand{\eeq}{\end{equation}}
\newcommand{\ben}{\begin{eqnarray}}
\newcommand{\een}{\end{eqnarray}}
\newcommand{\beno}{\begin{eqnarray*}}
\newcommand{\eeno}{\end{eqnarray*}}

\numberwithin{equation}{section}

\begin{document}

\title [Stability on the 3D MHD equtions]{Stability and large-time behavior on 3D incompressible MHD equations with partial dissipation near a background magnetic field}

\author[Lin, Wu, Zhu]{Hongxia Lin$^{1}$, Jiahong Wu$^{2}$ and Yi Zhu$^{3}$  }
	
	\address{$^{1}$ College of Mathematics and Physics,  and  Geomathematics Key Laboratory of Sichuan Province, Chengdu University of Technology, Chengdu 610059, P.R. China }
	
	\email{linhongxia18@126.com}

    \address{$^2$ Department of Mathematics, Oklahoma State University, Stillwater, OK 74078, United States}
	
	\email{jiahong.wu@okstate.edu}

    \address{$^{3}$ Department of Mathematics, East China University of Science and Technology, Shanghai, 200237, P. R. China}
	
	\email{Zhuyim@ecust.edu.cn}

\vskip .2in
\begin{abstract}
Physical experiments and numerical simulations have observed a remarkable stabilizing phenomenon:
a background magnetic field stabilizes and damps electrically conducting fluids. This paper intends 
to establish this phenomenon as a mathematically rigorous fact on a magnetohydrodynamic (MHD) system 
with anisotropic dissipation in $\mathbb R^3$. The velocity equation in this system is the 3D 
Navier-Stokes equation with dissipation only in the $x_1$-direction while the magnetic field obeys
the induction equation with magnetic diffusion in two horizontal directions. We establish that 
any perturbation near the background magnetic field $(0,1,0)$ is globally stable in the 
Sobolev setting $H^3(\R^3)$. In addition, explicit decay rates in $H^2(\R^3)$ are also obtained. 
When there is no presence of the magnetic field, the 3D anisotropic Navier-Stokes equation in $\mathbb R^3$
is not well understood and the small data global well-posedness   remains an intriguing open 
problem. This paper reveals the mechanism of how the magnetic field generates enhanced dissipation 
and helps stabilize the fluid. 
\end{abstract}
\maketitle



\def\aint{\dashint}

\def\arraystretch{2}
\def\eps{\varepsilon}

\def\s#1{\mathbb{#1}} 
\def\t#1{\tilde{#1}} 
\def\b#1{\overline{#1}}
\def\N{\mathcal{N}} 
\def\M{\mathcal{M}} 
\def\R{{\mathbb{R}}}

\def\ba{\begin{aligned}}
\def\ea{\end{aligned}}
\def\be{\begin{equation}}
\def\ee{\end{equation}}

\def\cz#1{\|#1\|_{C^{0,\alpha}}}
\def\ca#1{\|#1\|_{C^{1,\alpha}}}
\def\cb#1{\|#1\|_{C^{2,\alpha}}}

\def\lb#1{\|#1\|_{L^2}}
\def\ha#1{\|#1\|_{H^1}}
\def\hb#1{\|#1\|_{H^2}}

\def\cin{\subset\subset}
\def\Ld{\Lambda}
\def\ld{\lambda}
\def\ol{{\Omega_L}}
\def\sla{{S_L^-}}
\def\slb{{S_L^+}}
\def\C{\mathbf{C}} 
\def\cl#1{\overline{#1}}
\def\ra{\rightarrow}
\def\xra{\xrightarrow}
\def\g{\nabla}
\def\a{\alpha}
\def\b{\beta}
\def\d{\delta}
\def\th{\theta}
\def\fai{\varphi}
\def\O{\Omega}
\def\o{\omega}
\def\f{\frac}
\def\p{\partial}

\section{Introduction}

This paper deals with the stability and large-time behavior problem on a system of 3D anisotropic MHD equations near a background magnetic field. To shed some light on the potential 
difficulties of this problem, we briefly review several facts on the behavior of solutions to the Euler and the anisotropic Navier-Stokes equations.

\vskip .1in 
It is well-known that solutions of the incompressible Euler equations 
$$
\begin{cases}
	\p_t u + (u\cdot\na) u  =- \na P, \\
	\na\cdot u =0
\end{cases}
$$
can grow rather rapidly in time. In fact, Kiselev and Sverak are able to construct 
a vorticity solution of the 2D Euler equations in a disk whose gradient grows double exponentially in time \cite{Kis}. In the periodic setting, an example of Zlatos  
shows that the vorticity gradient can grow at least exponentially \cite{Zla}.  Choi and Jeong 
obtain linear in time growth for the vorticity gradient for certain smooth and compactly supported initial vorticity in $\mathbb R^2$ \cite{Jeong}.
Classical solutions to the 3D Euler equations could develop finite-time singularities (\cite{Chen,Elgindi}). Many more results in this direction can be found in a review paper by  Drivas and Elgindi \cite{Drivas}.
As a special consequence, perturbations governed by the Euler equations near the trivial solution
are generally not stable. How much dissipation does one really need in order to achieve the
stability? Adding the full Laplacian dissipation is certainly sufficient. As demonstrated by 
Schonbek and others (see, e.g., \cite{Sc0,Sc1, ScSc,Wie}), solutions of the Navier-Stokes equations 
$$
\begin{cases}
	\p_t u + (u\cdot\na) u  =- \na P + \mu \Delta u, \\
	\na\cdot u =0
\end{cases}
$$
are asymptotically stable and decay in time with explicit decay rates. When the dissipation is anisotropic and only in two directions, the Navier-Stokes equations become 
\beq\label{an}
\begin{cases}
	\p_t u + (u\cdot\na) u  =- \na P + \mu \Delta_h u, \\
	\na\cdot u =0.
\end{cases}
\eeq
where $\Delta_h =\p_1^2 + \p_2^2$ is the horizontal Laplacian. Due to its physical applications 
and special mathematical properties, (\ref{an}) has attracted considerable interests and an array of beautiful small data global well-posedness results have been obtained
(see, e.g., \cite{CDGG,CheZh,If,LiuZ, Pai1, Pai2, ZhangT1, ZhangT2}). 
New approaches have very recently been developed to tackle the large-time behavior problem and 
 explicit decay rates have been extracted for (\ref{an}) (see \cite{JiWuYang, XZ1}).  If we further reduce the dissipation to be in just one
direction, the resulting 3D anisotropic Navier-Stokes equations 
\beq\label{ns1}
\begin{cases}
	\p_t u + (u\cdot\na) u  =- \na P + \mu \p_1^2 u, \quad x\in \mathbb R^3, \,\,t>0,\\
	\na\cdot u =0
\end{cases}
\eeq
is not well-understood. In particular, the small data global well-posedness problem remains open.
In addition, very little is known on the stability properties and the large-time behavior. 

\vskip .1in 
This paper focuses on the following system of the 3D MHD equations with anisotropic dissipation
\begin{align}\label{MHD}
	\left\{
	\begin{array}{ll}
		\p_t u + (u\cdot\na) u  =- \na P + \mu \p_1^2u+(B\cdot\na) B, \ \ \ x\in{\R^3},\ t>0, \\
		\partial_{t}B+ (u\cdot\nabla) B=\eta\, \Delta_h B + (B\cdot\nabla) u,\\
		\nabla\cdot u=\nabla\cdot B=0
	\end{array}
	\right.
\end{align}
with the initial data
$$u(x,0)=u_0,\,\quad  B(x,0)=B_0. $$
Here $u=(u_1,u_2,u_3)^\top, B=(B_1, B_2, B_3)^\top$ and $P$ represent the velocity field of the fluid, the magnetic field and the scalar pressure, respectively. The constants $\mu>0$ and $\eta>0$ are the viscosity coefficient and the magnetic diffusivity. The MHD system (\ref{MHD})
focused here is relevant in the modeling of reconnecting plasmas (see, e.g., \cite{CLi1,CLi2,Pri}).

\vskip .1in 
The motivation for studying (\ref{MHD}) comes from two distinct sources. The first is the 
stabilizing phenomenon observed in physical experiments involving electrically conducting fluids.
The experiments exhibit a remarkable phenomenon: a background magnetic field actually stabilizes 
and damps turbulent MHD fluids (see, e.g., \cite{AMS, Alex, Alf,  Davi0, Davi1, DA, Gall, Gall2}). We intend to establish this phenomenon as a mathematically 
rigorous fact on (\ref{MHD}). The second is to initiate new strategies and develop innovative tools for stability and large-time behavior problems on anisotropic models.

\vskip .1in 
To understand the stabilizing mechanism of a background magnetic field
$$
u^{(0)} \equiv 0, \quad  B^{(0)} \equiv  e_2 := (0,1,0),
$$
which is obviously a steady-state of (\ref{MHD}), we study the dynamics of the perturbation 
$(u,b)$ with $b= B-B^{(0)}$.  Clearly $(u,b)$ satisfies the MHD equations
\begin{align}\label{MHD1}
	\left\{
	\begin{array}{ll}
		\p_t u + (u\cdot\na) u  =- \na P + \mu \p_1^2u+(b\cdot\na) b+\p_2b, \ \ \ x\in{\R^3},\ t>0, \\
		\partial_{t}b+ (u\cdot\nabla) b=\eta\, \Delta_h b + (b\cdot\nabla) u+\p_2u, \\
		\nabla\cdot u=\nabla\cdot b=0, \\
		u(x,0)=u_{0}(x), \qquad b(x,0) =b_0(x).
	\end{array}
	\right.
\end{align}
Our main result asserts the global well-posedness and stability of $(u,b)$, and provides precise 
decay rates for various Sobolev norms of $(u, b)$. The precise statement of these results is
presented in the following theorem. To simplify the notation, we use $\|f\|_{L^r L^q L^p}$ for  
the norm 
$\Big \|\,\,  \big\|\,\, \|f\|_{L^p(\R)}\big\|_{L^q(\R)} \Big \|_{L^r(\R)}$, and  $\|f\|_{L^q_{x_3}L^p_{x_1x_2}}$ for $\big\|\,\, \|f\|_{L^p_{x_1x_2}(\R^2)}\big\|_{L^q_{x_3}(\R)}$.

\begin{theorem}\label{main}
	Assume $(u_0,b_0)\in H^3(\R^3)$ with
	$\nabla\cdot u_0=0$ and $\nabla\cdot b_0=0$ satisfies
	$$
	(u_0,b_0),\,(\p_3u_0,\p_3b_0),\, (\p_3^2u_0,\p_3^2b_0)\in L^2_{x_3}L^1_{x_1x_2}(\R^3).
	$$
	Then there exists a sufficiently small constant $\delta >0$ such that, if
	\begin{align}\label{Ini}
		&\|(u_0,b_0)\|_{H^3(\R^3)}+\|(u_0,b_0)\|_{L^2_{x_3}L^1_{x_1x_2}(\R^3)}+\|(\p_3u_0,\p_3b_0)\|_{L^2_{x_3}L^1_{x_1x_2}(\R^3)}\nonumber\\
		&\quad +\|(\p_3^2u_0,\p_3^2b_0)\|_{L^2_{x_3}L^1_{x_1x_2}(\R^3)} \leq \delta,
	\end{align}
	then (\ref{MHD1}) admits a unique global solution $(u, b) \in C\big([0, \infty); H^3(\mathbb R^3)\big)$. In addition, $(u, b)$ is stable in the sense that, for an absolute constant $C>0$,
	\begin{align*}
		\|(u,b)(t)\|_{H^3(\R^3)}^2+\int_0^t\big(\|\p_1 u(\tau)\|_{H^3(\R^3)}^2+\|\p_2 u(\tau)\|_{H^2(\R^3)}^2+\|\na_h b(\tau)\|_{H^3(\R^3)}^2\big)\,d\tau\le C\delta^2
	\end{align*}
	for any $t>0$. 
	
	Furthermore, $(u,b)$ obeys the following time decay estimates, for $0<\varepsilon \le \f {1}{36}$, 
	\begin{align*}
		&\quad\|(u,b)\|_{L^2(\R^3)}\le C (1+t)^{-\f 12},\,\,
		\qquad \|(\na_h u,\na_h b)\|_{L^2(\R^3)}\le C (1+t)^{-1},\\
		&\quad \|(\p_3u,\p_3b)\|_{L^2(\R^3)}\le C (1+t)^{-\f 12+\varepsilon},
		\quad\|(\p_1\p_j u,\p_1\p_j b)\|_{L^2(\R^3)}\le C (1+t)^{-\f 54+\varepsilon},\,j=1, 2,\\
		&\quad \|(\p_1\p_3 u,\p_1\p_3 b)\|_{L^2(\R^3)}\le C (1+t)^{-1+\varepsilon},
		\quad \|(\p_2\p_j u,\p_2\p_j b)\|_{L^2(\R^3)}\le C (1+t)^{-\f 23+\varepsilon},\,j=2, 3,\\
		&\quad\|(\p_3^2 u,\p_3^2 b)\|_{L^2(\R^3)}\le C (1+t)^{-\f 14}.
	\end{align*}
\end{theorem}

\vskip .1in 
Theorem \ref{main} rigorously confirms the smoothing and stabilizing effect of the magnetic field 
on the electrically conducting fluids. Without the magnetic field, the fluid motion is governed
by the 3D anisotropic Navier-Stokes equation \eqref{ns1} alone and whether or not the velocity
is stable in Sobolev spaces remains an outstanding open problem. When coupled with magnetic field,
Theorem \ref{main} ensures that any perturbation near a background magnetic is stable and decays to 
zero at explicit rates as $t\to\infty$. 

\vskip .1in 
We clarify the differences between  Theorem \ref{main} and some of the closely related results. 
Wu and Zhu \cite{WuZhu} solved the stability problem for the MHD system with horizontal dissipation 
$\Delta_h u$ and vertical magnetic diffusion $\p_3^2 b$. It appears that the situation considered
here is more difficult. This is due to the handling of the velocity nonlinearity $(u\cdot\na) u$. 
When the velocity dissipation is only in one direction, the triple-product term 
$((u\cdot\na) u, u)_{H^3}$ is much more difficult to control than any triple product terms 
generated for the MHD system considered in \cite{WuZhu}. In fact, this term is exactly the reason 
why the well-posedness problem on the 3D anisotropic Navier-Stokes (\ref{ns1}) is open. 
One main contribution of this paper is the handling of the Navier-Stokes nonlinearity 
when the dissipation of the velocity is only in a single direction. The smoothing and stabilizing effect of the magnetic field on the fluids, and the elaborate construction of 
time-weighted energy functional are the key ingredients of this successful story. We remark that 
there is a very large mathematical literature on the incompressible MHD equations. In particular, there have been substantial recent developments on the well-posedness and stability problems, and 
significant progress has been made (see, e.g., \cite{AP,BLW, BSS, Bee, CL,CRW, CW,CWJ, ChenZhang, Deng, DLW,DLW1, DZ, Duv, Fef, Fef1, Feng, HeXuYu, HuX, HuLin, HuWang, JiWu, JNiu, Lai1, Lai2, LTY, LJW, LXZ, LZ, LZ1, LinDu, Liu0, PZZ, Ren1, Ren2, ST, Shang, Tan, Wan, WeiZ, WeiZ2, Wu2018, WuWu, WuWuXu, XZ, XZZ, Yam1,Yam2,YangWu,  YuanZhao, Zhang, Zhangt, Zhou}).

\vskip .1in 
We explain the proof of Theorem \ref{main}. Due to the lack of velocity dissipation in two
directions, we take the functional setting to be the Sobolev space $H^3$ in order to guarantee
the uniqueness. The local existence follows from a standard procedure (see, e.g., \cite{MaBe}), so
we focus on the global {\it a priori} bounds of $(u, b)$. This is accomplished via the bootstrapping
 argument (see, e.g., \cite{Tao}). A crucial step is to construct a suitable energy
functional. Naturally it should include the $H^3$-norm together with the time integral pieces 
from the dissipative terms
$$
	E^{(1)}_0(t)=\sup_{0\le\tau\le t}\|(u(\tau),b(\tau))\|_{H^3}^2+\int_0^t \Big(\|\p_1 u(\tau)\|_{H^3}^2+\|\nabla_h b(\tau)\|_{H^3}^2\Big)d \tau.
$$
However, due to the lack of velocity dissipation in two directions, the triple product generated by
the nonlinearity, namely $((u\cdot\na)u, u)_{H^3}$ can not be bounded in terms of $E^{(1)}_0(t)$. The
most difficult piece is the following triple product 
$$
\int \p_3^3(u\cdot\na u)\cdot \p_3^3 u\, dx. 
$$
Here we have used $\int$ to denote the integral in $x$ over $\mathbb R^3$. To distinguish the 
derivatives in different directions, we further write it as 
\begin{align}\label{Exa}
	&\int \p_3^3(u\cdot\na u)\cdot \p_3^3 u\, dx\nonumber\\
	&= 3\int \p_3 u_h \cdot\na_h\p_3^2 u \cdot \p_3^3 u\, dx + 3\int \p_3^2 u_h \cdot\na_h\p_3 u \cdot \p_3^3 u\, dx +\int \p_3^3 u_h \cdot \na_h u \cdot \p_3^3 u \, dx\nonumber\\
	&\quad+3\int \p_3 u_3 \,\p_3^3 u \cdot \p_3^3 u\, dx + 3\int \p_3^2 u_3 \,\p_3^2 u \cdot \p_3^3 u\, dx +\int \p_3^3 u_3 \, \p_3 u \cdot \p_3^3 u \, dx.
\end{align}
Clearly we need to seek enhanced dissipation in the $x_2$ or the $x_3$ direction to complement the 
existing dissipation in the $x_1$-direction. The background magnetic field is along the $x_2$ direction and it is in this direction that the extra regularization is generated. Mathematically 
this is reflected in the wave structure. We explain this. To avoid unnecessary complications, we 
look at the linearized system of (\ref{MHD1}), namely 
\begin{equation}\label{lin}
\begin{cases} 
	\p_t u = \mu \p_2^2 u + \p_2 b, \\
	\p_t b = \eta \Delta_h b + \p_2 u, \\
	\na\cdot u =\na\cdot b=0. 
\end{cases}
\end{equation}
By differentiating (\ref{lin}) in $t$ and making several substitutions, (\ref{lin}) can be converted 
into the following system of wave equations 
\begin{equation}\label{wv}
	\begin{cases} 
		\p_{tt} u - (\mu \p_2^2 + \eta \Delta_h) \p_t u + \mu\eta \p_2^2 \Delta_h u -\p_2^2u =0, \\
		\p_{tt} b - (\mu \p_2^2 + \eta \Delta_h) \p_t b + \mu\eta \p_2^2 \Delta_h b -\p_2^2b =0, \\
		\na\cdot u =\na\cdot b=0. 
	\end{cases}
\end{equation}
(\ref{wv}) is a system of anisotropic and degenerate wave equations. In comparison with (\ref{lin}),
(\ref{wv}) exhibits much more smoothing and stabilizing properties. In particular, the two terms
$\p_2^2u$ and $\p_2^2b$ in (\ref{wv}), emerged from the interaction of the velocity and the magnetic
field,  generates the dissipation in the $x_2$-direction. This confirms the stabilizing effect of the background magnetic field. To include this regularizing property in the energy functional, 
we define 
$$
E^{(2)}_0(t)=\int_0^t \|\p_2 u(\tau)\|_{H^2}^2d \tau.
$$
We emphasize that the extra dissipative effect in the $x_2$-direction is one-derivative lower than 
what a standard dissipation term $\p^2_2 u$ provides. This is why this energy functional only allows
the time integrability of $\|\p_2 u\|_{H^2}^2$, not  $\|\p_2 u\|_{H^3}^2$. 
Combining $E^{(1)}_0$ and $E^{(2)}_0$ gives 
\begin{align*}
E_0(t)&= E^{(1)}_0 + E^{(2)}_0\\
&=\sup_{0\le\tau\le t}\|(u(\tau),b(\tau))\|_{H^3}^2+\int_0^t \Big(\|\p_1 u(\tau)\|_{H^3}^2
+\|\p_2 u(\tau)\|_{H^2}^2+\|\nabla_h b(\tau)\|_{H^3}^2\Big)d \tau.
\end{align*}
However, there are still two terms in (\ref{Exa}) (the third term and the fourth term) that can 
not be bounded in terms of $E_0(t)$. After invoking the divergence-free condition 
$\p_3 u_3 =-\p_1 u_1-\p_2 u_2$, these terms are reduced to the difficult term 
\begin{align}\label{Di}
	\int |\p_2 u|\, |\p_3^3 u|\, |\p_3^3 u| dx.
\end{align}
Due to the aforementioned weaker smoothing effect in the $x_2$-direction, (\ref{Di}) can not be bounded by $E_0(t)$. Extra maneuvers are necessary.   Our idea  
is to include two extra time-weighted energy functionals 
\begin{align*}
	E_1(t)=&\sup_{0\le\tau\le t}(1+\tau)\|(\na_hu(\tau),\na_hb(\tau))\|_{H^1}^2\\
	&+\int_0^t(1+\tau) \Big(\|\p_1\na_h u(\tau)\|_{H^1}^2 +\|\p_2\na_h u(\tau)\|_{L^2}^2+\|\na_h^2 b(\tau)\|_{H^1}^2\Big)d \tau,\nonumber\\
	E_2(t)=&\sup_{0\le\tau\le t}\Big(
	(1+\tau)\|(u(\tau),b(\tau))\|_{L^2}^2+(1+\tau)^2\|(\na_h u(\tau),\na_h b(\tau))\|_{L^2}^2\\
	&+(1+\tau)^{1-2\varepsilon }\|(\p_3u(\tau),\p_3 b(\tau))\|_{L^2}^2\nonumber\\
	&+\sum_{j=1}^{2}(1+\tau)^{\f 52-2\varepsilon}\|(\p_1\p_j u(\tau), \p_1\p_j b(\tau))\|_{L^2}^2\\
	&+\sum_{j=2}^{3}(1+\tau)^{\f 43-2\varepsilon}\|(\p_2\p_j u(\tau), \p_2\p_j b(\tau))\|_{L^2}^2\nonumber\\
	&+(1+\tau)^{2-2\varepsilon }\|(\p_1\p_3 u(\tau), \p_1\p_3 b(\tau))\|_{L^2}^2
	+(1+\tau)^{\f 12 }\|(\p_3^2 u(\tau), \p_3^2 b(\tau))\|_{L^2}^2. 
	\Big).
\end{align*}
We shall show that the inclusion of $E_1(t)$ and $E_2(t)$ enables us to bound the term in (\ref{Di}) 
suitably and thus establish a closed energy inequality. The definition of $E_2(t)$ is certainly 
not simple. It takes into account of the precise time decay rate of each norm involved in $E_2(t)$. 
We will resort to the integral representation of (\ref{MHD1}) and spectral analysis to control the 
terms in $E_2(t)$. Having obtained the necessary components of the energy functional, we sum them 
up to form our total energy functional 
\beq
E(t)=E_0(t)+E_1(t)+E_2(t).\nonumber\\
\eeq
Our main efforts are devoted to proving the following estimate 
\begin{align}
	E(t)\le C_1F(u_0,b_0)+ C_2\Big(E^{\f 32}(t)+E^2(t)\Big),\label{E}
\end{align}
where $C_1$ and $C_2$ are constants, and 
\begin{align*}
	F(u_0,b_0)&=\|(u_0,b_0)\|_{H^3}^2+\|(u_0,b_0)\|_{L^2_{x_3}L^1_{x_1x_2}}^2
	+\|(\p_3u_0,\p_3b_0)\|_{L^2_{x_3}L^1_{x_1x_2}}^2+\|(\p_3^2u_0,\p_3^2b_0)\|_{L^2_{x_3}L^1_{x_1x_2}}^2.
\end{align*}
Verifying (\ref{E}) is a very lengthy process. For the sake of clarity, we divide 
the whole process into the proofs of the following inequalities
\begin{align}
& E_0(t)\le C E(0)+ C E^\f 32(t), \label{e0b}\\
& E_1(t)\le CE(0)+CE_0(t)+ C E^{\f 32}(t), \label{e1b} \\
&E_2(t)\le  C \Big(E^{\f 32}(t)+ E^2(t)\Big)
+C\Big(\|(u_0,b_0)\|_{H^2}^2+\|(u_0,b_0)\|_{L^2_{x_3}L^1_{x_1x_2}}^2\nonumber\\
&\qquad\quad  +\|(\p_3u_0,\p_3b_0)\|_{L^2_{x_3}L^1_{x_1x_2}}^2+\|(\p_3^2u_0,\p_3^2b_0)\|_{L^2_{x_3}L^1_{x_1x_2}}^2\Big). \label{e2b}
\end{align}
To prove (\ref{e0b}), we realize that $E_0(t)$ consists of two different types of terms 
$E^{(1)}_0(t)$ and $E^{(2)}_0(t)$, as aforementioned. The boundedness of $E^{(2)}_0(t)$ relies on
the enhanced dissipation from the wave structure. Naturally the proof of (\ref{e0b}) is further split into two parts, 
\begin{align*}
	(\|u(t)\|_{H^3}^2+\|b(t)\|_{H^3}^2)+2\int_0^t \Big(\mu\|\p_1u(\tau)\|_{H^3}^2+\eta\|\na_h b(\tau)\|_{H^3}^2\Big)d\tau
	\le CE(0)+ C E^{\f 32}(t)
\end{align*}
and 
\begin{align*}
	&-(\p_2u(t), b(t))_{H^2}+\f 12\int_0^t \|\p_2u(\tau)\|_{H^2}^2
	-\int_0^t \Big(\|\p_2b(\tau)\|_{H^2}^2+(\mu^2+\eta^2)\|\Delta_h b(\tau)\|_{H^2}^2\Big)d\tau\nonumber\\
	&\quad\le CE(0)+C E^{\f 32}(t).
\end{align*}
The detailed estimates are provided in Section \ref{sec3}. To prove (\ref{e1b}), we also need to 
divide the terms in $E_1(t)$ into two parts, 
$$
\int_0^t (1+ \tau) \|\p_2 \na_h u(\tau)\|_{L^2}^2\,d\tau 
$$
and the rest of the terms. The regularization from the wave structure in (\ref{wv}) is 
used to gain the time integrability of the vertical derivative. More technical 
details are left in Section \ref{sec4}.

\vskip .1in 
The proof of (\ref{e2b}) is extremely elaborate and relies on the precise decay rates of the norms 
involved in $E_2(t)$. Direct energy estimates are not sufficient for this purpose. Instead we solve 
the system of linear equations (\ref{lin}) and recast the nonlinear system (\ref{MHD1}) into an
integral form. This form relies on three kernel functions. They are degenerate and anisotropic in the
frequency space. We first perform a detailed spectral analysis in suitably divided subdomains of the
frequency space to obtain sharp and precise upper bounds for the kernel functions. The terms in 
$E_2(t)$ are then estimated according to the orders and directions of their derivatives. After a lengthy process, we finally obtain (\ref{e2b}). 

\vskip .1in 
Once (\ref{E}) is at our disposal, a direct application of the bootstrapping argument yields the 
desired global bounds and Theorem \ref{main} then follows. 

\vskip .1in 
The rest of this paper is divided into four sections. Section \ref{sec2} applies the bootstrapping
argument to the a priori inequality (\ref{E}) to establish Theorem \ref{main}. In addition, several 
anisotropic inequalities for products and triple products are provided here as well. 
They will be used in the subsequent sections. Section \ref{sec3} details the proof of (\ref{e0b}). 
Section \ref{sec4} proves  (\ref{e1b}) while Section \ref{sec5} is devoted to (\ref{e2b}).

\vskip .3in
\section{Proof of Theorem \ref{main} and anisotropic Sobolev inequalities}
\label{sec2}

This section serves two purposes. The first is to prove Theorem \ref{main} by applying 
the bootstrapping argument to the {\it a priori} inequality in (\ref{E}). The second is to provide
anisotropic inequalities for several products and triple products, which will be used in the proofs
in subsequent sections.

\begin{proof}[Proof of Theorem \ref{main}] The local (in time) well-posedness of (\ref{MHD1}) 
	in $H^3$ can be shown via standard procedures (see, e.g., \cite{MaBe}). It suffices to 
	establish the global bounds stated in Theorem \ref{main} in order to obtain the 
	global existence. This is accomplished by applying the bootstrapping argument to (\ref{E}), 
	namely
\begin{align}
E(t)\le C_1F(u_0,b_0)+ C_2\Big(E^{\f 32}(t)+E^2(t)\Big),\label{E1}
\end{align}
where
\begin{align*}
F(u_0,b_0)&=\|(u_0,b_0)\|_{H^3}^2+\|(u_0,b_0)\|_{L^2_{x_3}L^1_{x_1x_2}}^2
 +\|(\p_3u_0,\p_3b_0)\|_{L^2_{x_3}L^1_{x_1x_2}}^2\\
 &\quad +\|(\p_3^2u_0,\p_3^2b_0)\|_{L^2_{x_3}L^1_{x_1x_2}}^2.
\end{align*}
A useful description of the bootstrapping argument can be found in \cite[p.21]{Tao}. 
In order to apply the bootstrapping argument, we make the ansatz that
\begin{align}
E(t)\le M:=\min\Big\{1,\f {1}{(4C_2)^2}\Big\}. \label{ans}
\end{align}
We then verify that $E(t)$ actually admits a smaller bound, 
\begin{align*}
E(t)\le \f {M}{2}.
\end{align*}
Inserting (\ref{ans}) in (\ref{E1}) and recalling the initial assumption (\ref{Ini}), we have 
\begin{align*}
E(t)&\le C_1F(u_0,b_0)+C_2\Big( M^\f 12 + M \Big)E(t)\\
    &\le C_1 \d^2 + 2C_2M^{\f 12}E(t)\\
    &\le C_1 \d^2 + \f 1 2 E(t),
\end{align*}
or
\begin{align*}
E(t)\le 2C_1 \d^2.
\end{align*}
If the initial data is sufficiently small, say
$$
\d^2 \le \f {M} {4C_1},
$$
then we derive
\begin{align*}
E(t)\le 2C_1\d^2 \le  \f {M}{2}.
\end{align*}
The bootstrapping argument then implies $T=\infty$ and asserts that for any time $t>0$,
\begin{align*}
E(t)\le C\d^2,
\end{align*}
which, in particular, implies the desired global bound on the solution $(u,b)$. As a 
consequence, we obtain the global existence of solutions. The uniqueness is obvious due to the 
high regularity of the solution. The global bound on $E_2(t)$ yields the desired decay rates 
stated in Theorem \ref{main}.  This completes the proof of Theorem {\ref{main}}.
\end{proof}

\vskip .1in 
In the second part of this section, we provide several anisotropic upper bounds for products and 
triple products. The bounds stated in the following lemma are powerful tools in controlling the 
nonlinearity in terms of the anisotropic dissipation. 

\begin{lemma}\label{lem21}
For some constants $C>0$, $i,j,k=1,2,3$ and $i\neq j\neq k$, we have
\begin{align}
&\int|fgh|\,dx\le C\| f\|_{L^2(\R^3)}^{\f 12}\|\p_1 f\|_{L^2(\R^3)}^{\f 12}\| g\|_{L^2(\R^3)}^{\f 12}\|\p_2 g\|_{L^2(\R^3)}^{\f 12}
\|h\|_{L^2(\R^3)}^{\f 12}\|\p_3 h\|_{L^2(\R^3)}^{\f 12},\label{Ine1}\\
&\int|fgh|\,dx\le C\|f\|_{L^2(\R^3)}^{\f 14}\|\p_if\|_{L^2(\R^3)}^{\f 14} \|\p_jf\|_{L^2(\R^3)}^{\f 14} \|\p_i\p_jf\|_{L^2(\R^3)}^{\f 14}\notag\\
&\qquad\qquad\qquad \times \|g\|_{L^2(\R^3)}^{\f 12} \|\p_kg\|_{L^2(\R^3)}^{\f12}\|h\|_{L^2(\R^3)}\label{Ine2},\\
&\quad\|fg\|_{L^2(\R^3)}
\le C\|f\|_{L^2(\R^3)}^{\f 14}\|\p_if\|_{L^2(\R^3)}^{\f 14}\|\p_jf\|_{L^2(\R^3)}^{\f 14}\|\p_i\p_jf\|_{L^2(\R^3)}^{\f 14}
     \|g\|_{L^2(\R^3)}^{\f 12}\|\p_kg\|_{L^2(\R^3)}^{\f 12},\label{Ine3}\\
&\quad\|f\, g\|_{L_{x_3}^2 L_{x_1x_2}^1}\le  C\|f\|_{L^2(\R^3)}^{\f 12} \|\p_3f\|_{L^2(\R^3)}^{\f 12}  \|g\|_{L^2(\R^3)}.\label{fg}
\end{align}
\end{lemma}

\begin{proof}
The first two estimates have been stated and proven in \cite{WuZhu}. Here we give the proof of (\ref{Ine3}) and (\ref{fg}).
Without loss of generality, we assume $i=2,j=3,k=1$ in (\ref{Ine3}). Now we prove (\ref{Ine3}).  By H\"{o}lder's inequality, for $l=1,2,3$, we have the simple fact
\begin{align}\label{f}
\|f\|_{L_{x_l}^{\infty}(\R)} \leq \sqrt{2} \|f\|_{L_{x_l}^2(\R)}^{\f 12} \|\partial_{l}f\|_{L_{x_l}^2(\R)}^{\f 1 2}.
\end{align}
By (\ref{f}), 
\begin{align*}
\|fg\|_{L^2(\R^3)}&\le \big\|\, \|f\|_{L_{x_1}^2} \|g\|_{L_{x_1}^{\infty}}\, \big\|_{L_{x_2x_3}^2}\\
&\le C\big\|\, \|f\|_{L_{x_1}^2} \|g\|_{L_{x_1}^{2}}^{\f 12} \|\p_1g\|_{L_{x_1}^{2}}^{\f 12}\, \big\|_{L_{x_2x_3}^2}\\
&\le C\|f\|_{L_{x_2x_3}^{\infty}L_{x_1}^2} \|g\|_{L^{2}(\R^3)}^{\f 12} \|\p_1g\|_{L^{2}(\R^3)}^{\f 12}.
\end{align*}
By Minkowski's inequality, (\ref{f}) and H\"{o}lder's inequality, 
\begin{align*}
\|f\|_{L_{x_2x_3}^{\infty}L_{x_1}^2}&\le \big\| \,\|f\|_{L_{x_2}^{\infty}}\, \big\|_{L_{x_1}^2 L_{x_3}^{\infty}}
 \le C\big\| \,\|f\|_{L_{x_2}^{2}}^{\f 12} \|\p_2f\|_{L_{x_2}^{2}}^{\f 12} \,\big\|_{L_{x_1}^2 L_{x_3}^{\infty}}\\
& \le C \Big\| \, \big\| \, \|f\|_ {L_{x_3}^{\infty}}\big\|_{L_{x_2}^{2}}^{\f 12} \, \big\|\, \|\p_2f\|_{ L_{x_3}^{\infty}}\big\|_{L_{x_2}^{2}}^{\f 12} \,\Big\|_{L_{x_1}^2}\\
& \le C\big\|\, \|f\|_{L_{x_3}^{\infty}} \, \big\|_{L_{x_1x_2}^2}^{\f 12}
       \big\| \, \|\p_2f\|_{L_{x_3}^{\infty}} \, \big\|_{L_{x_1x_2}^2}^{\f 12}\\
& \le C\|f\|_{L^2(\R^3)}^{\f 14}\|\p_3f\|_{L^2(\R^3)}^{\f 14}\|\p_2f\|_{L^2(\R^3)}^{\f 14}\|\p_2\p_3f\|_{L^2(\R^3)}^{\f 14}.
\end{align*}
Therefore, 
\begin{align*}
\|fg\|_{L^2(\R^3)}
\le C\|f\|_{L^2(\R^3)}^{\f 14}\|\p_2f\|_{L^2(\R^3)}^{\f 14}\|\p_3f\|_{L^2(\R^3)}^{\f 14}\|\p_2\p_3f\|_{L^2(\R^3)}^{\f 14}
     \|g\|_{L^2(\R^3)}^{\f 12}\|\p_1g\|_{L^2(\R^3)}^{\f 12}.
\end{align*}
To prove (\ref{fg}), we apply H\"{o}lder's inequality, Minkowski's inequality and (\ref{f}) to obtain
 \begin{align*}
 \|f\, g\|_{L_{x_3}^2L_{x_1x_2}^1}
 &\le C \Big\| \|f\|_{L_{x_1x_2}^2}\|g\|_{L_{x_1x_2}^2}\Big\|_{L_{x_3}^2}
 \le C\Big\|\, \|f\|_{L_{x_3}^{\infty}}\Big \|_{L_{x_1x_2}^2}\|g\|_{L^2(\R^3)}\\
 &\le  C\|f\|_{L^2(\R^3)}^{\f 12} \|\p_3f\|_{L^2(\R^3)}^{\f 12}  \|g\|_{L^2(\R^3)}.
 \end{align*}
This completes the proof of Lemma \ref{lem21}. 
\end{proof}

\vskip .3in
\section{Estimate for $E_0(t)$ }
\label{sec3}

This section is devoted to proving the {\it a priori} estimate (\ref{e0b}) for $E_0(t)$. More precisely, we prove the following proposition. We exploit the extra smoothing reflected
in the wave structure (\ref{wv}) to make up for the lack of vertical dissipation in the velocity
equation. The idea is to consider a  Lyapunov functional involving an inner product besides the 
standard $H^3$-norm.

\begin{proposition}\label{pro1}
Let $(u,b)$ be a solution of the system (\ref{MHD1}). Then, for some constant $C>0$, we have
\begin{align}\label{E0}
E_0(t)\le C E(0)+ C E^\f 32(t).
\end{align}
\end{proposition}

\vskip .1in
To prove (\ref{E0}), we work with the Lyapunov functional defined by
$$
L(u, b)(t) =\|(u(t),b(t))\|_{H^3}^2 + \lambda \big(\p_2 u(t), b(t)\big)_{H^2},
$$
where $\lambda>0$ is a small parameter. Next we show the bound of $L(u,b)$. We evaluate the 
time evolution of each part in this Lyapunov functional. For the sake of clarity, we divide 
this process into two lemmas. The first focuses on bounding $\|(u(t),b(t))\|_{H^3}^2$ while the second handles the inner product
$\big(\p_2 u(t), b(t)\big)_{H^2}$.

\vskip .1in
\begin{lemma}\label{lem32}
Assume $(u, b)$ is a solution to (\ref{MHD1}). Then we have
\begin{align*}
(\|u(t)\|_{H^3}^2+\|b(t)\|_{H^3}^2)+2\int_0^t \Big(\mu\|\p_1u(\tau)\|_{H^3}^2+\eta\|\na_h b(\tau)\|_{H^3}^2\Big)d\tau
\le CE(0)+ C E^{\f 32}(t).
\end{align*}
\end{lemma}

\begin{proof}[Proof of Lemma \ref{lem32}]
 First we take the $L^2$-inner product of (\ref{MHD1}) with $(u,b)$ to obtain 
\begin{align}\label{ub0}
\f 12\f {d}{dt}(\|u(t)\|^2_{L^2}+\|b(t)\|^2_{L^2})+\big(\mu\|\partial_{1} u\|^{2}_{L^2}+\eta\|\na_h b\|^{2}_{L^2}\big)=0.
\end{align}
 Due to the equivalence of the norm $\|(u(t),b(t))\|_{H^{3}}$  with $\|(u(t),b(t))\|_{L^2}+\|(u(t),b(t))\|_{\dot{H}^{3}},$  it suffices to bound $\|(u(t),b(t))\|_{\dot{H}^{3}}$.
Applying $\partial^3_i\,(i=1,2,3)$ to the equations (\ref{MHD1}) and taking the $L^2$-inner product of the resulting equations with $(\partial^3_iu,\partial^3_ib)$, we have
\begin{align}\label{I0}
&\frac{1}{2}\sum_{i = 1}^3 \frac{d}{dt} \Big(\|\partial_i^3 u(t)\|_{L^2}^2 +  \|\partial_i^3 b(t)\|_{L^2}^2\Big)
+\sum_{i = 1}^3\Big( \mu\|\partial_i^3 \partial_1 u\|_{L^2}^2+ \eta\|\partial_i^3 \na_h b\|_{L^2}^2\Big) \nonumber\\
&=-\sum_{i = 1}^3\int \partial_i^3 (u\cdot \nabla u)\cdot \partial_i^3 u \; dx
+\sum_{i = 1}^3\int \partial_i^3 (b \cdot \nabla b)\cdot \partial_i^3 u\; dx \nonumber\\
&\quad- \sum_{i = 1}^3\int \partial_i^3 (u \cdot \nabla b)\cdot \partial_i^3 b \; dx
+\sum_{i = 1}^3\int \partial_i^3 (b \cdot \nabla u)\cdot \partial_i^3 b \; dx\nonumber\\
&:= I_1 + I_2 + I_3 + I_4.
\end{align}
By Leibniz formula, integration by parts and $\na\cdot u =0$, we have
\begin{align*}
I_1& = -\sum^2_{i=1} \sum_{k = 1}^3 \mathcal{C}_{3}^k\int\partial_i^k u \cdot \na\partial_i^{3-k} u\cdot \partial_i^3 u \, dx
- \sum_{k = 1}^3 \mathcal{C}_{3}^k\int\partial_3^k u \cdot \na\partial_3^{3-k} u\cdot \partial_3^3 u \, dx \\
&:=I_{11}+I_{12},
\end{align*}
where $\mathcal{C}_{3}^k$ is the standard binomial coefficient. 
By H\"{o}lder's inequality and Sobolev's inequality,
\begin{align}\label{I11}
I_{11}&=-\sum_{i = 1}^2\Big(3\int\partial_i u \cdot \na\partial_i^{2} u\cdot \partial_i^3 u \, dx
+3\int\partial_i^2 u \cdot \na\partial_i u\cdot \partial_i^3 u \, dx
+\int\partial_i^3 u \cdot \na u\cdot \partial_i^3 u \, dx\Big)\nonumber\\
&\le C (\|\na u\|_{L^\infty}\|\na\na_h^2u\|_{L^2}+\|\na_h^2u\|_{L^4}\|\na_h\na u\|_{L^4})\|\na_h^3u\|_{L^2}\nonumber\\
&\le C(\|\na u\|_{H^2}\|\na\na_h^2u\|_{L^2}+\|\na_h^2u\|_{H^1}\|\na_h\na u\|_{H^1})\|\na_h^3u\|_{L^2}\nonumber\\
&\le C\|\na u\|_{H^2}\|\na_h^2 u\|_{H^1}^2.
\end{align}
Rewriting the terms $I_{12}$ in components, we have
\begin{align*}
I_{12}&\le 4\int |\p_3u|\, |\p_3^2\na_h u|\,|\p_3^3u|\, dx  + 6\int |\na_h\p_3u|\, |\p_3^2 u|\,|\p_3^3u|\, dx\\
&\quad-3 \int \p_3u_3 \, \p_3^3u \cdot \p_3^3u \, dx   -\int \p_3^3u_h\cdot \na_h u\cdot \p_3^3u \, dx\\
&\le 4\int |\p_3u|\, |\p_3^2\na_h u|\,|\p_3^3u|\, dx   +6\int |\na_h\p_3u|\, |\p_3^2 u|\,|\p_3^3u|\, dx\\
&\quad+ 8 \int |u|\, |\p_3^3u|\,|\p_1\p_3^3u|\, dx   + 4\int |\p_2 u|\, |\p_3^3u|^2 \, dx\\
&:=I_{121}+I_{122} +I_{123}+I_{124},
\end{align*}
where we have used  the divergence-free condition, $\p_3 u_3 = -\p_1 u_1 -\p_2 u_2$. 
By the anisotropic inequalities (\ref{Ine1}) and (\ref{Ine2}), 
\begin{align*}
I_{121}+I_{122}&\le C \|\p_3u\|_{L^2}^{\f 14}\|\p_2\p_3u\|_{L^2}^{\f 14}\|\p_3^2u\|_{L^2}^{\f 14}\|\p_3^2\p_2u\|_{L^2}^{\f 14}
    \|\p_3^2\na_h u\|_{L^2}  \|\p_3^3u\|_{L^2}^{\f 12}\|\p_3^3\p_1u\|_{L^2}^{\f 12}\\
& \quad+ C\|\na_h\p_3u\|_{L^2}^{\f 12} \|\na_h\p_3^2u\|_{L^2}^{\f 12}   \|\p_3^2u\|_{L^2}^{\f 12}\|\p_2\p_3^2u\|_{L^2}^{\f 12}
    \|\p_3^3u\|_{L^2}^{\f 12}\|\p_1\p_3^3u\|_{L^2}^{\f 12}\\
&\le C \|\na u\|_{H^2}(\|\na_h\na u\|_{H^1}^2+\|\p_1\na^3u\|_{L^2}^2).
\end{align*}
Applying (\ref{Ine2}) again, $I_{123}$ can be bounded by
\begin{align*}
I_{123}
&\le C \|u\|_{L^2}^{\f 14}\|\p_2u\|_{L^2}^{\f 14}\|\p_3u\|_{L^2}^{\f 14}\|\p_2\p_3u\|_{L^2}^{\f 14}
        \|\p_3^3u\|_{L^2}^{\f 12}   \|\p_1\p_3^3 u\|_{L^2}^{\f 32} \\
&\le  C  \|u\|_{H^3}(\|\p_2 u\|_{H^1}^2+\|\p_1\na^3u\|_{L^2}^2).
\end{align*}
Therefore, 
\begin{align}\label{I12}
I_{12}&\le C \|u\|_{H^3}(\|\na_hu\|_{H^2}^2+\|\p_1\na^3u\|_{L^2}^2)+I_{124},
\end{align}
where $I_{124}$ will be estimated at the end of the proof.
Consequently, (\ref{I11}), together with (\ref{I12}), leads to
\begin{align}\label{I1}
I_1&\le C \|u\|_{H^3}(\|\na_hu\|_{H^2}^2+\|\p_1\na^3u\|_{L^2}^2)+I_{124}.
\end{align}
Since $b$ has better dissipation than $u$, it is simpler to bound $I_2$. By Leibniz's formula, 
\begin{align*}
I_2& = \sum^2_{i=1} \sum_{k = 1}^3 \mathcal{C}_{3}^k\int\partial_i^k b \cdot \na\partial_i^{3-k} b\cdot \partial_i^3 u \, dx
+ \sum_{k = 1}^3 \mathcal{C}_{3}^k\int\partial_3^k b \cdot \na\partial_3^{3-k} b\cdot \partial_3^3 u \, dx\\
&\quad+ \int b \cdot \nabla \p_i^3b \cdot \p_i^3u \, dx\\
&:=I_{21}+I_{22}+ \int b \cdot \nabla \p_i^3b \cdot \p_i^3u \, dx.
\end{align*}
As in $I_{11}$, we first have
\begin{align}\label{I21}
I_{21}
&\le C (\|\na b\|_{L^\infty}\|\na\na_h^2b\|_{L^2}+\|\na_h^2b\|_{L^4}\|\na_h\na b\|_{L^4})\|\na_h^3u\|_{L^2}\nonumber\\
&\le C(\|\na b\|_{H^2}\|\na\na_h^2b\|_{L^2}+\|\na_h^2b\|_{H^1}\|\na_h\na b\|_{H^1})\|\na_h^3u\|_{L^2}\nonumber\\
&\le C\|\na b\|_{H^2}(\|\na_h^2 b\|_{H^1}^2+\|\na_h^3u\|_{L^2}^2).
\end{align}
For $I_{22}$, we further split it into two parts and then apply (\ref{Ine1}) to get
\begin{align}\label{I22}
I_{22}&= \sum_{k = 1}^3 \mathcal{C}_{3}^k\int\partial_3^k b_h \cdot \na_h\partial_3^{3-k} b\cdot \partial_3^3 u \, dx
     +   \sum_{k = 1}^3 \mathcal{C}_{3}^k \int\partial_3^k b_3 \, \partial_3^{4-k} b\cdot \partial_3^3 u \, dx\nonumber\\
&\le C\sum_{k = 1}^3\|\partial_3^k b_h\|_{L^2}^{\f 12} \|\p_2\partial_3^k b_h\|_{L^2}^{\f 12}
     \|\na_h\partial_3^{3-k} b\|_{L^2}^{\f 12}\|\p_3\na_h\partial_3^{3-k} b\|_{L^2}^{\f 12}
    \|\p_3^3u\|_{L^2}^{\f 12}\|\p_1\p_3^3u\|_{L^2}^{\f 12}\nonumber\\
&\quad + C\sum_{k = 1}^3\|\partial_3^k b_3\|_{L^2}^{\f 12} \|\partial_3^{k+1} b_3\|_{L^2}^{\f 12}
     \|\partial_3^{4-k} b\|_{L^2}^{\f 12}\|\p_2\partial_3^{4-k} b\|_{L^2}^{\f 12}
    \|\p_3^3u\|_{L^2}^{\f 12}\|\p_1\p_3^3u\|_{L^2}^{\f 12}\nonumber\\
&\le C (\|\na b\|_{H^2}+\|\na^3 u\|_{L^2})(\|\na_h b\|_{H^3}^2+\|\p_1\p_3^3u\|_{L^2}^2).
\end{align}
Therefore, (\ref{I21}) and  (\ref{I22}) yield
\begin{align}\label{I2}
I_2\le C (\|\na b\|_{H^2}+\|\na^3 u\|_{L^2})(\|\na_h b\|_{H^3}^2+\|\na_h^3 u\|_{L^2}^2+\|\p_1\na^3u\|_{L^2}^2)
+ \int b \cdot \nabla \p_i^3b \cdot \p_i^3u \, dx.
\end{align}
We proceed to deal with $I_3$.  $I_3$ is firstly divided into three parts,
\begin{align*}
I_3& = -\sum^2_{i=1} \sum_{k = 1}^3 \mathcal{C}_{3}^k\int\partial_i^k u \cdot \na\partial_i^{3-k} b\cdot \partial_i^3 b\, dx
- \sum_{k = 1}^3 \mathcal{C}_{3}^k\int\partial_3^k u_h \cdot \na_h\partial_3^{3-k} b\cdot \partial_3^3 b \, dx \\
&- \sum_{k = 1}^3 \mathcal{C}_{3}^k\int\partial_3^k u_3 \, \partial_3^{4-k} b\cdot \partial_3^3 b \, dx \\
&:=I_{31}+I_{32}+I_{33}.
\end{align*}
By (\ref{Ine1}), 
\begin{align}\label{I3132}
I_{31}+I_{32}
&\le C\sum^2_{i=1}\sum_{k = 1}^3\|\partial_i^k u\|_{L^2}^{\f 12} \|\p_1 \partial_i^ku\|_{L^2}^{\f 12}
     \|\na\partial_i^{3-k} b\|_{L^2}^{\f 12}\|\p_2\na\partial_i^{3-k} b\|_{L^2}^{\f 12}
    \|\p_i^3b\|_{L^2}^{\f 12}\|\p_3\p_i^3b\|_{L^2}^{\f 12}\nonumber\\
 & \quad+C\sum_{k = 1}^3\|\partial_3^k u_h\|_{L^2}^{\f 12} \|\p_1\partial_3^k u_h\|_{L^2}^{\f 12}
     \|\na_h\partial_3^{3-k} b\|_{L^2}^{\f 12}\|\p_3\na_h\partial_i^{3-k} b\|_{L^2}^{\f 12}
    \|\p_3^3b\|_{L^2}^{\f 12}\|\p_2\p_3^3b\|_{L^2}^{\f 12}\nonumber\\
&\le C(\|\na u\|_{H^2}+\|\na b\|_{H^2})(\|\p_1\na u\|_{H^2}^2+\|\na_h b\|_{H^3}^2).
\end{align}
For $I_{33}$, we further decompose it, integrate by parts and use (\ref{Ine2}) to get
\begin{align}\label{I33}
I_{33}&=- \sum_{k = 2}^3 \mathcal{C}_{3}^k\int\partial_3^k u_3\, \partial_3^{4-k} b\cdot \partial_3^3 b \, dx
       +6 \int u_h\cdot \p_3^2\na_h b\cdot \p_3^3 b\, dx\nonumber\\
&\le  C \sum_{k = 2}^3 \|\p_3^k u_3\|_{L^2}\|\partial_3^{4-k} b\|_{L^2}^{\f 14}\|\p_1\partial_3^{4-k} b\|_{L^2}^{\f 14}
        \|\p_3\partial_3^{4-k} b\|_{L^2}^{\f 14}\|\p_1\p_3\partial_3^{4-k} b\|_{L^2}^{\f 14}
         \|\p_3^3b\|_{L^2}^{\f 12} \|\p_2\p_3^3b\|_{L^2}^{\f 12}\nonumber\\
    &\quad+ C \|u\|_{L^2}^{\f 14}\|\p_1u\|_{L^2}^{\f 14} \|\p_3u\|_{L^2}^{\f 14}\|\p_1\p_3u\|_{L^2}^{\f 14}
         \|\p_3^2\na_h b\|_{L^2}    \|\p_3^3b\|_{L^2}^{\f 12} \|\p_2\p_3^3b\|_{L^2}^{\f 12}\nonumber\\
&\le C (\|u\|_{H^1}+\|\na b\|_{H^2})(\|\p_1 u\|_{H^1}^2+\|\na\na_h u\|_{H^1}^2+\|\na\na_h b\|_{H^2}^2),
\end{align}
where we have used $\na\cdot u=0$. Combining (\ref{I3132}) and (\ref{I33}) yields
\begin{align}\label{I3}
I_3\le C (\| u\|_{H^3}+\|\na b\|_{H^2})(\|\na_h\na u\|_{H^1}^2+\|\p_1 u\|_{H^3}^2+\|\na_h b\|_{H^3}^2).
\end{align}
We now bound $I_4$. As in $I_2$, we decompose $I_4$ into three parts, 
\begin{align*}
I_4& = \sum^2_{i=1} \sum_{k = 1}^3 \mathcal{C}_{3}^k\int\partial_i^k b \cdot \na\partial_i^{3-k} u\cdot \partial_i^3 b \, dx
+ \sum_{k = 1}^3 \mathcal{C}_{3}^k\int\partial_3^k b \cdot \na\partial_3^{3-k} u\cdot \partial_3^3 b \, dx\\
&\quad+ \int b \cdot \nabla \p_i^3u \cdot \p_i^3b \, dx\\
&:=I_{41}+I_{42}+ \int b \cdot \nabla \p_i^3u \cdot \p_i^3b\, dx.
\end{align*}
By H\"{o}lder's inequality and Sobolev's inequality, 
\begin{align*}
I_{41}
\le C  \sum_{k = 1}^3 \|\na_h^kb\|_{L^4}\|\na\na_h^{3-k}u\|_{L^2}\|\na_h^3 b\|_{L^4}
\le C\|\na u\|_{H^2}\|\na_h b\|_{H^3}^2.
\end{align*}
The estimate for $I_{42}$ is more subtle. We first further split it into three terms,
\begin{align*}
I_{42}
&=  3\int\partial_3 b_h \cdot \na_h\partial_3^2 u\cdot \partial_3^3 b \, dx
 +\sum_{k = 2}^3 \mathcal{C}_{3}^k\int\partial_3^k b_h \cdot \na_h\partial_3^{3-k} u\cdot \partial_3^3 b \, dx\\
 &\quad+\sum_{k = 1}^3 \mathcal{C}_{3}^k\int\partial_3^k b_3 \, \partial_3^{4-k} u\cdot \partial_3^3 b \, dx\\
 &:=I_{421}+I_{422}+I_{423}.
\end{align*}
Applying (\ref{Ine2}) to $I_{421}$,  and (\ref{Ine1}) to $I_{422}$ and $I_{423}$, respectively, we obtain
\begin{align*}
I_{421}
&\le  C  \|\partial_3 b_h\|_{L^2}^{\f 14} \|\p_2\partial_3 b_h\|_{L^2}^{\f 14}
         \|\p_3^2 b_h\|_{L^2}^{\f 14} \|\p_2\p_3^2 b_h\|_{L^2}^{\f 14}
         \|\na_h\p_3^2 u\|_{L^2}     \|\p_3^3b\|_{L^2}^{\f 12} \|\p_1\p_3^3b\|_{L^2}^{\f 12}\nonumber\\
&\le C \|\na b\|_{H^2}(\|\na^2\na_h u\|_{L^2}^2+\|\na\na_h b\|_{H^2}^2),
\end{align*}
and
\begin{align*}
I_{422}+I_{423}
&\le C\sum_{k = 2}^3\|\partial_3^k b_h\|_{L^2}^{\f 12} \|\p_1\partial_3^k b_h\|_{L^2}^{\f 12}
     \|\na_h\partial_3^{3-k} u\|_{L^2}^{\f 12}\|\p_3\na_h\partial_3^{3-k} u\|_{L^2}^{\f 12}
    \|\partial_3^3 b\|_{L^2}^{\f 12}\|\p_2\partial_3^3 b\|_{L^2}^{\f 12}\nonumber\\
 & +C\sum_{k = 1}^3\|\partial_3^k b_3\|_{L^2}^{\f 12} \|\partial_3^{k+1} b_3\|_{L^2}^{\f 12}
     \|\partial_3^{4-k} u\|_{L^2}^{\f 12}\|\p_1\partial_3^{4-k} u\|_{L^2}^{\f 12}
    \|\p_3^3b\|_{L^2}^{\f 12}\|\p_2\p_3^3b\|_{L^2}^{\f 12}\nonumber\\
&\le C(\|\na u\|_{H^2}+\|\na^2 b\|_{H^1})(\|\p_1\na u\|_{H^2}^2+\|\na_h u\|_{H^2}^2+\|\na_h b\|_{H^3}^2).
\end{align*}
Thus,
\begin{align}\label{I4}
I_4\le C (\|\na u\|_{H^2}+\| \na b\|_{H^2})(\|\na_h u\|_{H^2}^2+\|\p_1\na u\|_{H^2}^2+\|\na_h b\|_{H^3}^2).
\end{align}
Inserting (\ref{I1}), (\ref{I2}), (\ref{I3}) and (\ref{I4}) in (\ref{I0}) and combining with (\ref{ub0}),  we conclude
\begin{align}\label{I0b}
&\frac{1}{2} \frac{d}{dt} \Big[\|(u(t),b(t))\|_{L^2}^2
+\sum_{i = 1}^3\|(\partial_i^3 u(t),\p_i^3b(t))\|_{L^2}^2 \Big]
+\Big[ \mu\|\p_1u\|_{L^2}^2+\eta\|\na_h u\|_{L^2}^2 \nonumber\\
&\quad +\sum_{i = 1}^3(\mu\|\partial_i^3 \partial_1 u\|_{L^2}^2
+ \eta\|\partial_i^3 \na_h b\|_{L^2}^2)\Big] \nonumber\\
&\le C (\|u\|_{H^3}+\|  b\|_{H^3})(\|\p_2u\|_{H^2}^2+\|\p_1 u\|_{H^3}^2+\|\na_h b\|_{H^3}^2)+I_{124}.
\end{align}
Integrating (\ref{I0b}) over $[0,t]$ yields
\begin{align*}
&\|(u(t),b(t))\|_{H^3}^2+2\int_0^t\big(\mu\|\p_1u (\tau)\|_{H^3}^2+\eta\|\na_h b (\tau)\|_{H^3}^2\big)\,d \tau \nonumber\\
& \le C \int_0^t (\|u(\tau)\|_{H^3}+\|  b(\tau)\|_{H^3})(\|\p_2 u(\tau)\|_{H^2}^2+\|\p_1 u(\tau)\|_{H^3}^2+\|\na_h b(\tau)\|_{H^3}^2)d\tau \nonumber\\
&\quad + C(\|u_0\|_{H^3}^2+\|b_0\|_{H^3}^2) +C\int_0^t I_{124}(\tau)\, d\tau  \nonumber\\
&\le CE_0^{\f 32}(t)+CE(0)+C\int_0^t I_{124}(\tau)\, d\tau.
\end{align*}
It remains to bound the integral of $I_{124}$. By means of (\ref{Ine2}), we have
\begin{align*}
I_{124}\le C \|\p_2u\|_{L^2}^{\f 14}\|\p_2^2u\|_{L^2}^{\f 14}\|\p_3\p_2u\|_{L^2}^{\f 14}\|\p_2^2\p_3u\|_{L^2}^{\f 14}
            \|\p_3^3u\|_{L^2}^{\f 32}\|\p_1\p_3^3u\|_{L^2}^{\f 12}.
\end{align*}
Then applying H\"{o}lder's inequality leads to
\begin{align*}
\int_0^t I_{124}(\tau) d\tau
&\le C\sup_{0\le \tau \le t} (1+\tau)^{\f 14} \|\p_2u(\tau)\|_{L^2}^{\f 14}\,
                           (1+\tau)^{\f 14(\f 23-\varepsilon )}\|\p_2\p_3u(\tau)\|_{L^2}^{\f 14}
                           \|\p_3^3u(\tau)\|_{L^2}^{\f 32}\\
& \quad \times \int_0^t   (1+\tau)^{\f 18}\|\p_2^2 u(\tau)\|_{L^2}^{\f 14} \,  \|\p_2^2\p_3u(\tau)\|_{L^2}^{\f 14}
                  \|\p_1\p_3^3u(\tau)\|_{L^2}^{\f 12}\, (1+\tau)^{-\f{13}{24}+\f {\varepsilon}{4}} d \tau\\
& \le C E_2^{\f 14}(t)  E_0^{\f 34}(t)  \Big(\int_0^t (1+\tau)\|\p_2^2u(\tau)\|_{L^2}^2d \tau\Big)^{\f 18} \,
                                  \Big(\int_0^t \|\p_2^2\p_3u(\tau)\|_{L^2}^{2}d\tau\Big)^{\f 18}\\
                                &\quad \times\Big(\int_0^t\|\p_1\p_3^3u(\tau)\|_{L^2}^2d\tau\Big)^{\f 14}
                                  \Big(\int_0^t(1+\tau)^{-\f{13}{12}+\f {\varepsilon}{2}} d\tau\Big)^{\f 12}\\
&\le C  E_2^{\f 14}(t)   E_1^{\f 18}(t)  E_0^{\f 9 8}(t)
\le C E^{\f 3 2}(t).
\end{align*}
Therefore, 
\begin{align*}
\|(u(t),b(t))\|_{H^3}^2+2\int_0^t(\mu\|\p_1u (\tau)\|_{H^3}^2+\eta\|\na_h b (\tau)\|_{H^3}^2)\,d \tau
\le C E^{\f 32}(t) + CE(0).
\end{align*}
This completes the proof of Lemma \ref{lem32}.
\end{proof}

\vskip .1in
Next we evaluate the inner product $(\p_2u(t), b(t))_{H^2}$ and prove the following lemma.

\begin{lemma}\label{lem33}
Assume $(u, b)$ is a solution to (\ref{MHD1}). Then 
\begin{align}\label{2uH2}
&-(\p_2u(t), b(t))_{H^2}+\f 12\int_0^t \|\p_2u(\tau)\|_{H^2}^2
-\int_0^t \Big(\|\p_2b(\tau)\|_{H^2}^2+(\mu^2+\eta^2)\|\Delta_h b(\tau)\|_{H^2}^2\Big)d\tau\nonumber\\
&\quad\le CE(0)+C E^{\f 32}(t).
\end{align}
\end{lemma}

\begin{proof}[Proof of Lemma \ref{lem33}]
Invoking the equations of $u$ and $b$ in (\ref{MHD1}), we have 
\begin{align}\label{D2uH2}
&-\f {d}{dt}(\p_2u(t), b(t))_{H^2}+ \|\p_2u\|_{H^2}^2- \|\p_2b\|_{H^2}^2\nonumber\\
&=(\p_2(u\cdot\na u), b)_{H^2}-(\p_2(b\cdot\na b), b)_{H^2}+(\p_2u, u\cdot\na b)_{H^2}-(\p_2u, b\cdot\na u)_{H^2}\nonumber\\
& \quad -\mu(\p_2\p_1^2u, b)_{H^2}-\eta(\p_2u, \Delta_h b)_{H^2}\nonumber\\
&:= I_5+\cdots+I_{10}.
\end{align}
By integration by parts, $I_5$ can be rewritten as
\begin{align*}
 I_5&=-\int u\cdot\na u\cdot(\p_2b-\p_2\Delta b)\,dx+\int\na (u\cdot\na u)\cdot\p_2\na^3b\,dx\nonumber\\
    &=-\int u\cdot\na u\cdot (\p_2b-\p_2\Delta b)\,dx+\int(\na u\cdot\na) u\cdot\p_2\na^3b\,dx\\
    &\quad +\int (u\cdot \na)\na u\cdot\p_2\na^3b\,dx.
\end{align*}
Applying (\ref{Ine1}) and (\ref{Ine2}) leads to
\begin{align*}
I_{5}
&\le C\|u\|_{L^2}^{\f 12} \|\p_2u\|_{L^2}^{\f 12}  \|\na u\|_{L^2}^{\f 12}\|\p_1\na u\|_{L^2}^{\f 12}
    \|\p_2b+\p_2\Delta b\|_{L^2}^{\f 12}\|\p_3\p_2b+\p_3\p_2\Delta b\|_{L^2}^{\f 12}\nonumber\\
& +  C\|\na u\|_{L^2}^{\f 14} \|\p_1\na u\|_{L^2}^{\f 14}  \|\p_3\na u\|_{L^2}^{\f 14}\|\p_1\p_3\na u\|_{L^2}^{\f 14}
    \|\na u\|_{L^2}^{\f 12}\|\p_2\na u\|_{L^2}^{\f 12}\|\p_2\na^3 b\|_{L^2}\nonumber\\
& +  C\| u\|_{L^2}^{\f 14} \|\p_1 u\|_{L^2}^{\f 14}  \|\p_3  u\|_{L^2}^{\f 14}\|\p_1\p_3 u\|_{L^2}^{\f 14}
    \|\na^2 u\|_{L^2}^{\f 12}\|\p_2\na^2 u\|_{L^2}^{\f 12}\|\p_2\na^3 b\|_{L^2}\nonumber\\
&\le C\| u\|_{H^2}(\|\na_h u\|_{H^2}^2+\|\na_h b\|_{H^3}^2).
\end{align*}
Similarly,
\begin{align*}
I_{6} \le C\| b\|_{H^2}\|\na_h b\|_{H^3}^2.
\end{align*}
For $I_7$, we split it into two parts
\begin{align*}
 I_7=\int u\cdot\na b\cdot(\p_2u-\p_2\Delta u)\,dx+\int \Delta(u\cdot\na b)\cdot \p_2\Delta u\,dx
    := I_{71}+I_{72}.
\end{align*}
By (\ref{Ine2}), 
\begin{align*}
 I_{71}&\le C  \| u\|_{L^2}^{\f 14} \|\p_1 u\|_{L^2}^{\f 14}  \|\p_3  u\|_{L^2}^{\f 14}\|\p_1\p_3 u\|_{L^2}^{\f 14}
              \|\na b\|_{L^2}^{\f 12}\|\p_2\na b\|_{L^2}^{\f 12}\|\p_2u+\p_2 \Delta u \|_{L^2}\nonumber\\
       &\le C(\| u\|_{H^1}+\|\na b\|_{L^2}) (\|\na_h u\|_{H^2}^2+\|\p_2\na b\|_{L^2}^2).
\end{align*}
Similarly, making use of the inequality (\ref{Ine2}) again, we get
\begin{align*}
I_{72}&=\int (\Delta u\cdot\na b+2\na u\cdot\na^2 b+  u\cdot\na\Delta b )\cdot \p_2\Delta u\,dx\nonumber\\
&\le  C\|\Delta u\|_{L^2}^{\f 12} \|\p_1\Delta u\|_{L^2}^{\f 12}
       \|\na b\|_{L^2}^{\f 14}\|\p_2\na b\|_{L^2}^{\f 14}\|\p_3\na b\|_{L^2}^{\f 14}\|\p_2\p_3\na b\|_{L^2}^{\f 14}
       \|\p_2\Delta u\|_{L^2}\nonumber\\
&\quad +  C \| \na u\|_{L^2}^{\f 14} \|\p_2\na u\|_{L^2}^{\f 14}  \|\p_3  \na u\|_{L^2}^{\f 14}\|\p_2\p_3\na u\|_{L^2}^{\f 14}
    \|\na^2 b\|_{L^2}^{\f 12}\|\p_1\na^2 b\|_{L^2}^{\f 12}\|\p_2\Delta u\|_{L^2}\nonumber\\
&\quad +  C \| u\|_{L^2}^{\f 14} \|\p_2u\|_{L^2}^{\f 14}  \|\p_3 u\|_{L^2}^{\f 14}\|\p_2\p_3 u\|_{L^2}^{\f 14}
    \|\na\Delta b\|_{L^2}^{\f 12}\|\p_1\na\Delta b\|_{L^2}^{\f 12}\|\p_2\Delta u\|_{L^2}\nonumber\\
&\le C(\| u\|_{H^2}+\|\na b\|_{H^2})(\|\na_h u\|_{H^2}^2+\|\na_h\na b\|_{H^2}^2),
\end{align*}
which, together with the estimate of $I_{71}$, gives
\begin{align*}
I_{7} \le C(\| u\|_{H^2}+\|\na b\|_{H^2})(\|\na_h u\|_{H^2}^2+\|\na_h\na b\|_{H^2}^2).
\end{align*}
$I_8$ can be estimated with the same process as $I_7$. Firstly,
\begin{align*}
 I_8=-\int b\cdot\na u\cdot(\p_2u-\p_2\Delta u)dx-\int \Delta(b\cdot\na u)\cdot \p_2\Delta u\,dx
    = I_{81}+I_{82}.
\end{align*}
Then we can derive
\begin{align*}
 I_{81}\le C(\|\na u\|_{L^2}+\| b\|_{H^1}) (\|\p_2u\|_{H^2}^2+\|\p_1 b\|_{H^1}^2).
\end{align*}
and
\begin{align*}
I_{82}&=-\int (\Delta b\cdot\na u+2\na b\cdot\na^2 u+  b\cdot\na\Delta u )\cdot \p_2\Delta u\,dx\nonumber\\
&\le  C\|\Delta b\|_{L^2}^{\f 12} \|\p_1\Delta b\|_{L^2}^{\f 12}
       \|\na u\|_{L^2}^{\f 14}\|\p_2\na u\|_{L^2}^{\f 14}\|\p_3\na u\|_{L^2}^{\f 14}\|\p_2\p_3\na u\|_{L^2}^{\f 14}
       \|\p_2\Delta u\|_{L^2}\nonumber\\
& +  C \| \na b\|_{L^2}^{\f 14} \|\p_2\na b\|_{L^2}^{\f 14}  \|\p_3  \na b\|_{L^2}^{\f 14}\|\p_2\p_3\na b\|_{L^2}^{\f 14}
    \|\na^2 u\|_{L^2}^{\f 12}\|\p_1\na^2 u\|_{L^2}^{\f 12}\|\p_2\Delta u\|_{L^2}\nonumber\\
& +  C \| b\|_{L^2}^{\f 14} \|\p_2b\|_{L^2}^{\f 14}  \|\p_3 b\|_{L^2}^{\f 14}\|\p_2\p_3 b\|_{L^2}^{\f 14}
    \|\na\Delta u\|_{L^2}^{\f 12}\|\p_1\na\Delta u\|_{L^2}^{\f 12}\|\p_2\Delta u\|_{L^2}\nonumber\\
&\le C(\|\na u\|_{H^2}+\| b\|_{H^2})(\|\p_2\na u\|_{H^1}^2+\|\p_1\na^2 u\|_{H^1}^2+\|\na_h b\|_{H^2}^2).
\end{align*}
Thus,
\begin{align*}
I_{8} \le C(\|\na u\|_{H^2}+\| b\|_{H^2})(\|\p_2 u\|_{H^2}^2+\|\p_1\na^2 u\|_{H^1}^2+\|\na_h b\|_{H^2}^2).
\end{align*}
By H\"{o}lder's inequality and Young's inequality, 
\begin{align*}
I_{9}+I_{10} =-\mu(\p_2u, \p_1^2b)_{H^2}-\eta(\p_2u, \Delta_h b)_{H^2}
\le \f 12\|\p_2 u\|_{H^2}^2+\mu^2\|\p_1^2b \|_{H^2}+\eta^2\|\Delta_h b \|_{H^2}^2.
\end{align*}
In summary, we have obtained
\begin{align}\label{2ub}
&-\f {d}{dt}(\p_2u(t), b(t))_{H^2}+\f 12\|\p_2u\|_{H^2}^2
- \Big(\|\p_2b\|_{H^2}^2+(\mu^2+\eta^2)\|\Delta_h b\|_{H^2}^2\Big)\nonumber\\
&\le C(\| u\|_{H^3}+\|b\|_{H^3})(\|\p_2 u\|_{H^2}^2+\|\p_1u\|_{H^3}^2+\|\na_h b\|_{H^3}^2).
\end{align}
Then integrating (\ref{2ub}) leads to the desired estimate (\ref{2uH2}). This completes the proof of Lemma \ref{lem33}.
\end{proof}

\vskip .1in
Now we ready to prove Proposition \ref{pro1}.
\begin{proof}[Proof of Proposition \ref{pro1}]
According to Lemma \ref{lem32} and \ref{lem33}, we have
\begin{align*}
&\big(\|u(t)\|_{H^3}^2+\|b(t)\|_{H^3}^2-\lambda(\p_2u(t), b(t))_{H^2}\big)
+\int_0^t \Big[2\mu\|\p_1u(\tau)\|_{H^3}^2\\
&\quad+\big(2\eta-\lambda(1+\mu^2+\eta^2)\big)\|\na_h b(\tau)\|_{H^3}^2
+\f \lambda 2\|\p_2u(\tau)\|_{H^2}^2 \Big]d\tau\\
&\le CE(0)+ C E^{\f 32}(t),
\end{align*}
where $\lambda$ is a parameter.
Now we select $\lambda$ to be sufficiently small to obtain
\begin{align*}
&\|u(t)\|_{H^3}^2+\|b(t)\|_{H^3}^2+\int_0^t \Big(\|\p_1 u(\tau)\|_{H^3}^2
+\|\p_2 u(\tau)\|_{H^2}^2+\|\nabla_h b(\tau)\|_{H^3}^2\Big)d \tau\\
&\le C E(0)+ C E^\f 32(t).
\end{align*}
This completes the proof of Proposition \ref{pro1}.
\end{proof}

\vskip .3in
\section{Estimate for $E_1(t)$}
\label{sec4}

The section proves the {\it a priori} inequality (\ref{e1b}) for $E_1(t)$. 
That is, we establish the following proposition. Since the velocity equation does not have the 
vertical dissipation, we need to make use of the extra smoothing and stabilization revealed 
by the wave structure in (\ref{wv}). Our idea is to use the inner product  
$(1+t)(\p_2\na_hu, \na_hb)$ to decode this regularizing property. As a consequence, we obtain 
the time integrability of  $(1+t)\,\|\p_2\na_hu\|_{L^2}^2$. More details are given in 
Lemma \ref{lem43} and its proof. 

\begin{proposition}\label{pro2}
For some constants $C>0$, it holds
\begin{align}\label{E1}
E_1(t)\le CE(0)+CE_0(t)+ C E^{\f 32}(t).
\end{align}
\end{proposition}

\vskip .1in
We shall divided the proof of (\ref{E1}) into two main parts.   The first one bounds 
the time-weighted energy $(1+t)\|(\na_h u, \na_h b)\|_{H^1}^2$ while the second handles 
the inner product $(1+t)(\p_2\na_hu, \na_hb)$ to generate the time-weighted dissipation $(1+t)\,\|\p_2\na_hu\|_{L^2}^2$.

\vskip .1in
\begin{lemma}\label{lem42}
Assume $(u, b)$ solves (\ref{MHD1}). Then we have
\begin{align}\label{Ghub}
&(1+t)\big(\|\na_h u(t)\|_{H^1}^2+\|\na_h b(t)\|_{H^1}^2\big)+2\int_0^t (1+\tau)\big(\mu\|\p_1\na_hu(\tau)\|_{H^1}^2+\eta\|\Delta_h b(\tau)\|_{H^1}^2\big)d\tau\nonumber\\
&\qquad\le E_0(t)+E(0)+C E^{\f 32}(t).
\end{align}
\end{lemma}

\begin{proof}[Proof of Lemma \ref{lem42}]
Taking the $H^1$-inner product of (\ref{MHD1}) with $(\Delta_h u, \Delta_h b)$, and multiplying by $(1+t)$, we obtain
\begin{align}\label{J0}
&\f 12 \f {d}{dt}(1+t)(\|\na_h u(t)\|_{H^1}^2+\|\na_h b(t)\|_{H^1}^2) + (1+t) (\mu\|\p_1\na_h u\|_{H^1}^2+\eta\|\Delta_h b\|_{H^1}^2)\nonumber\\
&=\f 12 (\|\na_h u\|_{H^1}^2+\|\na_h b\|_{H^1}^2) - (1+t)\big(\na_h(u\cdot \na u), \na_hu\big)_{H^1} +(1+t)\big(\na_h(b\cdot \na b), \na_hu\big)_{H^1}\nonumber\\
&\quad-(1+t)\big(\na_h(u\cdot \na b), \na_h b\big)_{H^1}+(1+t)\big(\na_h(b\cdot \na u), \na_h b\big)_{H^1}\nonumber\\
&:=\f 12 (\|\na_h u\|_{H^1}^2+\|\na_h b\|_{H^1}^2)+J_1+J_2+J_3+J_4.
\end{align}
To bound $J_1$, we split $J_1$ into three parts
\begin{align*}
J_1&=-(1+t)\Big(\int\na_h(u\cdot \na u)\cdot\na_hu\,dx +\int\na_h^2(u\cdot \na u)\cdot\na_h^2 u\,dx\\
&\quad \qquad\qquad  + \int\na_h\p_3(u\cdot \na u)\cdot\na_h\p_3u\,dx\Big)\\
 &:=-(1+t)(J_{11}+J_{12}+J_{13}).
\end{align*}
By the anisotropic inequality (\ref{Ine1}), 
\begin{align}\label{J11}
J_{11}&=\int\na_h u_h\cdot \na_h u\cdot\na_hu\,dx+\int\na_h u_3\, \p_3 u\cdot\na_hu\,dx\nonumber\\
    &\le C\|\na_h u\|_{L^2}^{\f 12}  \|\p_2\na_h u\|_{L^2}^{\f 12}
     \|\na_h u\|_{L^2}^{\f 12} \|\p_3\na_h u\|_{L^2}^{\f 12}
    \|\na_h u\|_{L^2}^{\f 12}\|\p_1\na_h u\|_{L^2}^{\f 12}\nonumber\\
& \quad+ C\|\na_h u_3\|_{L^2}^{\f 12}  \|\p_3\na_h u_3\|_{L^2}^{\f 12}
     \|\p_3 u\|_{L^2}^{\f 12} \|\p_2\p_3 u\|_{L^2}^{\f 12}
    \|\na_h u\|_{L^2}^{\f 12}\|\p_1\na_h u\|_{L^2}^{\f 12}\nonumber\\
&\le C \|\na_h u\|_{L^2} \|\na_h^2 u \|_{L^2}  \|\na_h u\|_{H^1}
    + C\|\na_h u\|_{L^2}^{\f 12} \|\p_3 u\|_{L^2}^{\f 12} \|\na_h^2 u\|_{L^2}  \|\na_h u\|_{H^1}.
\end{align}
Therefore,
\begin{align}\label{IJ11}
\int_0^t (1+\tau)&J_{11}(\tau) d\tau
\le C\sup_{0\le \tau\le t} (1+\tau)^{\f 12}\|\na_h u(\tau)\|_{L^2}
\int_0^t (1+\tau)^{\f 12}\|\na_h^2 u(\tau)\|_{L^2}
     \|\na_h u(\tau)\|_{H^1} d\tau\nonumber\\
& + C \sup_{0\le \tau\le t} (1+\tau)^{\f 12}\|\na_h u(\tau)\|_{L^2}^\f 12 \|\p_3 u(\tau)\|_{L^2}^{\f 12}
     \int_0^t (1+\tau)^{\f 12}\|\na_h^2 u(\tau)\|_{L^2}
     \|\na_h u(\tau)\|_{H^1} d\tau\nonumber\\
&\le C  E_1^{\f12}(t)E_1^{\f12}(t) E_0^{\f12}(t)
      + E_2^{\f14}(t) E_0^{\f14}(t) E_1^{\f12}(t) E^{\f12}_0(t)\nonumber\\
&\le C E^{\f 32}(t).
\end{align}
Applying (\ref{Ine1}) again and using Sobolev's inequality, $J_{12}$ can be bounded as
\begin{align}\label{J12}
J_{12}&=\int\na_h^2 u\cdot \na u\cdot\na_h^2u\,dx+2\int\na_h u \cdot \na \na_h u\cdot\na_h^2u\,dx\nonumber\\
    &\le  \|\na u\|_{L^{\infty}}\|\na_h^2u\|_{L^2}^2
    + C\|\na_h u\|_{L^2}^{\f 12}  \|\p_2\na_h u\|_{L^2}^{\f 12}
     \|\na\na_h u\|_{L^2}^{\f 12} \|\p_3\na\na_h u\|_{L^2}^{\f 12}
    \|\na_h^2 u\|_{L^2}^{\f 12}\|\p_1\na_h^2 u\|_{L^2}^{\f 12}\nonumber\\
  &\le C\|\na u\|_{H^2}\|\na_h^2u\|_{L^2}^2
    + C\|\na_h u\|_{L^2}^{\f 12}  \|\na_h^2 u\|_{L^2}
    \|\na\na_h u\|_{L^2}^{\f 12} \|\na_h\na^2 u\|_{L^2}.
\end{align}
Thus,
\begin{align}\label{IJ12}
&\int_0^t (1+\tau)J_{12}(\tau) d\tau
 \le C\sup_{0\le \tau\le t} \|\na u(\tau)\|_{H^2} \int_0^t (1+\tau)\|\na_h^2 u(\tau)\|_{L^2}^2 d\tau\nonumber\\
&\qquad + C \sup_{0\le \tau\le t} (1+\tau)^{\f 12}\|\na_h u(\tau)\|_{L^2}^\f 12  \|\na\na_h u(\tau)\|_{L^2}^\f 12
     \int_0^t (1+\tau)^{\f 12}\|\na_h^2 u(\tau)\|_{L^2}
     \|\na_h\na^2 u(\tau)\|_{L^2} d\tau\nonumber\\
&\quad\le C  E_0^{\f12}(t) E_1(t)
     + C E_2^{\f14}(t) E_0^{\f14}(t) E_1^{\f12}(t) E_0^{\f12}(t)\nonumber\\
&\quad\le C E^{\f 32}(t).
\end{align}
The bound for $J_{13}$ is more complicated. We first decompose it as follows,
 \begin{align*}
J_{13}&=\int\na_h\p_3 u\cdot \na u\cdot\na_h\p_3u\,dx+\int\na_h u \cdot \na \p_3 u\cdot\na_h\p_3u dx
        + \int\p_3 u \cdot \na \na_h u\cdot\na_h\p_3u\,dx\\
&\le 3\int |\na_h u|\, |\p_3\na_h u|^2 \, dx  +2\int |\p_3 u|\, |\na_h^2u| \, |\p_3\na_h u| \, dx
      +\int |\na_h u_3|\,|\p_3^2 u| \, |\p_3\na_h u| \, dx\\
&:= J_{131} + J_{132} +J_{133}.
\end{align*}
By means of (\ref{Ine1}) and (\ref{Ine2}), 
\begin{align}
J_{131}
&\le     C\|\na_h u\|_{L^2}^{\f 12}  \|\p_2\na_h u\|_{L^2}^{\f 12}
        \|\p_3\na_h u\|_{L^2}^{\f 12} \|\p_1\p_3\na_h u\|_{L^2}^{\f 12}
        \|\p_3\na_h u\|_{L^2}^{\f 12} \|\p_3^2\na_h u\|_{L^2}^{\f 12}\nonumber\\
&\le     C \|\na_h u\|_{L^2}^{\f 12}  \|\p_2\na_h u\|_{L^2}^{\f 12}
        \|\p_3\na_h u\|_{L^2}^{\f 12}  \|\p_3\na_h u\|_{H^1}^{\f 32} ,\label{J131}\\
J_{132}&\le
        C\|\p_3 u\|_{L^2}^{\f 14}  \|\p_2\p_3 u\|_{L^2}^{\f 14} \|\p_3^2 u\|_{L^2}^{\f 14} \|\p_2\p_3^2 u\|_{L^2}^{\f 14}
        \|\na_h^2 u\|_{L^2} \|\p_3\na_h u\|_{L^2}^{\f 12} \|\p_1\p_3\na_h u\|_{L^2}^{\f 12}\nonumber\\
        &\le C \|\p_3 u\|_{H^1}^{\f 12}\|\p_2\p_3 u\|_{H^1}^{\f 12}\|\na_h^2 u\|_{L^2}
           \|\p_3\na_h u\|_{L^2}^{\f 12} \|\p_1\p_3\na_h u\|_{L^2}^{\f 12},  \label{J132}
 \end{align}
and
\begin{align}
J_{133}
&\le     C\|\na_h u_3\|_{L^2}^{\f 12}  \|\p_3\na_h u_3\|_{L^2}^{\f 12}
        \|\p_3^2 u\|_{L^2}^{\f 12} \|\p_2\p_3^2 u\|_{L^2}^{\f 12}
        \|\p_3\na_h u\|_{L^2}^{\f 12} \|\p_1\p_3\na_h u\|_{L^2}^{\f 12}\nonumber\\
&\le    C\|\na_h u\|_{L^2}^{\f 12}   \|\na_h^2 u\|_{L^2}^{\f 12}
        \|\p_3^2 u\|_{L^2}^{\f 12}   \|\p_3\na_h u\|_{H^1}   \|\p_1\p_3\na_h u\|_{L^2}^{\f 12}.\label{J133}
 \end{align}
Thereby,  applying  H\"{o}lder's inequality gives 
\begin{align}
\int_0^t (1+\tau) J_{131}(\tau) d \tau
 &\le C\sup _{0\le \tau\le t} (1+\tau)^{\f 12}\|\na_h u(\tau)\|_{L^2}^{\f 12} (1+\tau)^{\f 14} \|\p_3\na_h u(\tau)\|_{L^2}^{\f 12}\nonumber\\
       &\qquad\times \int_0^t  (1+\tau)^{\f 14}\|\p_2\na_h u (\tau)\|_{L^2}^{\f 12}  \|\p_3\na_h u(\tau)\|_{H^1}^{\f 32}\,d\tau \nonumber\\
 & \le C   E_2^{\f 14}(t) E_1^{\f 12}(t) E_0^{\f 34}(t)  \le C E^{\f 3 2}(t),\label{IJ131}\\
\int_0^t (1+\tau) J_{132}(\tau) d \tau
 &\le C\sup _{0\le \tau\le t} \|\p_3 u(\tau)\|_{H^1}^{\f 12} (1+\tau)^{\f 14} \|\p_3\na_h u(\tau)\|_{L^2}^{\f 12}
        \int_0^t  (1+\tau)^{\f 12}\|\na_h^2 u (\tau)\|_{L^2}\nonumber\\
      &\qquad\times(1+\tau)^{\f 14}\|\p_3\p_1\na_h u(\tau)\|_{L^2}^{\f 12}
        \|\p_2\p_3 u(\tau)\|_{H^1}^{\f 12}\,d\tau\label{IJ132}\\
 & \le C   E_0^{\f 14}(t) E_1(t) E_0^{\f 1 4}(t)  \le C E^{\f 32}(t),\nonumber\\
\int_0^t (1+\tau) J_{133}(\tau) d \tau
 &\le C\sup _{0\le \tau\le t}  (1+\tau)^{\f 12} \|\na_h u(\tau)\|_{L^2}^{\f 12} \|\p_3^2 u(\tau)\|_{L^2}^{\f 12}
        \int_0^t  (1+\tau)^{\f 14}\|\na_h^2 u (\tau)\|_{L^2}^{\f 12}\nonumber\\
      &\qquad\times(1+\tau)^{\f 14}\|\p_3\p_1\na_h u(\tau)\|_{L^2}^{\f 12}
        \|\p_3\na_h u(\tau)\|_{H^1}\,d\tau \nonumber\\
 & \le C   E_2^{\f 14}(t) E_0^{\f 14}(t) E_1^{\f 1 2}(t) E_0^{\f 1 2}(t)
 \le C E^{\f 32}(t).\label{IJ133}
\end{align}
Adding (\ref{IJ131}), (\ref{IJ132}) and (\ref{IJ133}) yields
\begin{align}\label{IJ13}
\int_0^t (1+\tau) J_{13}(\tau) d \tau \le C  E^{\f 32}(t).
\end{align}
Consequently, according to the estimates (\ref{IJ11}), (\ref{IJ12}) and (\ref{IJ13}), we derive
\begin{align}\label{J1}
\int_0^t  J_{1}(\tau) d \tau \le C E^{\f 32}(t).
\end{align}
In the following, we handle $J_3$. The terms $J_2$ and $J_4$ will be estimated together later.
Firstly,
\begin{align*}
J_3&=-(1+t)\Big(\int\na_h(u\cdot \na b)\cdot\na_hb\,dx +\int\na_h^2(u\cdot \na b)\cdot\na_h^2 b\,dx\\
&\qquad\qquad\quad  + \int\na_h\p_3(u\cdot \na b)\cdot\na_h\p_3b\,dx\Big)
\\
   &:=-(1+t)(J_{31}+J_{32}+J_{33}).
\end{align*}
Invoking (\ref{J11}) and (\ref{J12}),
we have
\begin{align*}
J_{31}&=\int\na_h u_h\cdot \na_h b\cdot\na_hb\,dx+\int\na_h u_3\, \p_3 b\cdot\na_hb\,dx\\
    &\le C\|\na_h u\|_{L^2}^{\f 12}  \|\p_2\na_h u\|_{L^2}^{\f 12}
     \|\na_h b\|_{L^2}^{\f 12} \|\p_3\na_h b\|_{L^2}^{\f 12}
    \|\na_h b\|_{L^2}^{\f 12}\|\p_1\na_h b\|_{L^2}^{\f 12}\\
& \quad+ C\|\na_h u_3\|_{L^2}^{\f 12}  \|\p_3\na_h u_3\|_{L^2}^{\f 12}
     \|\p_3 b\|_{L^2}^{\f 12} \|\p_2\p_3 b\|_{L^2}^{\f 12}
    \|\na_h b\|_{L^2}^{\f 12}\|\p_1\na_h b\|_{L^2}^{\f 12},\\
J_{32}&=\int\na_h^2 u\cdot \na b\cdot\na_h^2b\,dx+2\int\na_h u \cdot \na \na_h b\cdot\na_h^2b\,dx\\
    & \le C\|\na b\|_{H^2}\|\na_h^2u\|_{L^2}\|\na_h^2b\|_{L^2}
    + C\|\na_h u\|_{L^2}^{\f 12}  \|\p_2\na_h u\|_{L^2}^{\f 12}
     \|\na\na_h b\|_{L^2}^{\f 12} \|\p_3\na\na_h b\|_{L^2}^{\f 12}
    \|\na_h^2 b\|_{L^2}^{\f 12}\|\p_1\na_h^2 b\|_{L^2}^{\f 12}.
\end{align*}
Then a similar argument to (\ref{IJ11}) and (\ref{IJ12}) gives
\begin{align*}
\int_0^t (1+\tau)(J_{31}+J_{32})(\tau) d\tau
&\le  C E^{\f 32}(t).
\end{align*}
For $J_{33}$, we still reformulate it into several integrals
 \begin{align*}
J_{33}
&\le 2\int |\na_h u|\, |\p_3\na_h b|^2 \, dx  + \int |\na_h b|\, |\p_3\na_h u_h| |\p_3\na_h b| \, dx\\
&+   \int |\p_3 b|\, |\na_h^2u| \, |\p_3\na_h b| \, dx + \int |\p_3 u|\, |\na_h^2b| \, |\p_3\na_h b| \, dx
+   \int |\na_h u_3| \,|\p_3^2 b|\, |\p_3\na_h b|  \, dx\\
&:= 2J_{331} + \cdots +J_{335}.
\end{align*}
Going through a similar process as in $J_{13}$, we are able to establish the bound for $J_{33}$. Recalling (\ref{J131}), we have
\begin{align*}
&J_{331}
\le     C\|\na_h u\|_{L^2}^{\f 12}  \|\p_2\na_h u\|_{L^2}^{\f 12}
        \|\p_3\na_h b\|_{L^2}^{\f 12} \|\p_1\p_3\na_h b\|_{L^2}^{\f 12}
        \|\p_3\na_h b\|_{L^2}^{\f 12} \|\p_3^2\na_h b\|_{L^2}^{\f 12}.
 \end{align*}
Then
\begin{align*}
&\int_0^t (1+\tau) J_{331}(\tau) d \tau
 \le C\sup _{0\le \tau \le t} (1+\tau)^{\f 12}\|\na_h u(\tau)\|_{L^2}^{\f 12} (1+\tau)^{\f 14} \|\p_3\na_h b(\tau)\|_{L^2}^{\f 12}\nonumber\\
  &\qquad\times\int_0^t  (1+\tau)^{\f 14}\|\p_2\na_h u (\tau)\|_{L^2}^{\f 12}  \|\p_3\na_h b(\tau)\|_{H^1}^{\f 32}\, d\tau
 \le C  E^{\f 32}(t).
\end{align*}
As in (\ref{J132}) and (\ref{IJ132}), $J_{333}$ can be bounded by
\begin{align*}
\int_0^t (1+\tau) J_{333}(\tau) d \tau
&\le C\sup _{0\le \tau \le t} \|\p_3 b(\tau)\|_{H^1}^{\f 12} (1+\tau)^{\f 14} \|\p_3\na_h b(\tau)\|_{L^2}^{\f 12}
        \int_0^t  (1+\tau)^{\f 12}\|\na_h^2 u (\tau)\|_{L^2}\nonumber\\
      &\times(1+\tau)^{\f 14}\|\p_3\p_1\na_h b(\tau)\|_{L^2}^{\f 12}
        \|\p_3\p_2 b(\tau)\|_{H^1}^{\f 12}\, d\tau
  \le  C E^{\f 3 2}(t).
\end{align*}
Also, from (\ref{J133}) and (\ref{IJ133}), we get
\begin{align*}
&\int_0^t (1+\tau) J_{335}(\tau) d \tau
 \le C\sup _{0\le \tau \le t}  (1+\tau)^{\f 12} \|\na_h u(\tau)\|_{L^2}^{\f 12} \|\p_3^2 b(\tau)\|_{L^2}^{\f 12}
        \int_0^t  (1+\tau)^{\f 14}\|\na_h\p_3 u_3 (\tau)\|_{L^2}^{\f 12}\nonumber\\
      &\times(1+\tau)^{\f 14}\|\p_3\p_1\na_h b(\tau)\|_{L^2}^{\f 12}
        \|\p_3\na_h b(\tau)\|_{H^1}\, d\tau
  \le  C  E^{\f 32}(t).
\end{align*}
The rest terms $J_{332}$ and $J_{334}$ can be handled as $J_{331}$ and $J_{333}$, respectively. Thus, we have
\begin{align*}
\int_0^t (1+\tau) (J_{332}(\tau)+J_{334}(\tau))\, d \tau
 \le  C E^{\f 32}(t).
\end{align*}
Consequently, we derive
\begin{align*}
\int_0^t (1+\tau) J_{33}(\tau)d \tau
 \le C E^{\f 32}(t).
\end{align*}
Combining all estimates above for $J_{31}$ through $J_{33}$, we conclude
\begin{align}\label{J3}
\int_0^t  J_{3}(\tau)d \tau
 \le C E^{\f 32}(t).
\end{align}

\vskip .1in
Finally we bound $J_{2}$ and $J_{4}$.  $J_{2}$ and $J_{4}$ can be estimated with a nearly same argument as $J_1$ and $J_3$, respectively. We shall just sketch the proof. By integration by parts and the divergence-free condition,
we split $J_{2}$ and $J_{4}$ into three parts as follows.
\begin{align*}
J_2+J_4:=(J_{21}+J_{22}+J_{23})(1+t),
\end{align*}
where
\begin{align*}
J_{21}&=\int (\na_h b \cdot\na b\cdot \na_h u +\na_h b \cdot\na u\cdot \na_h b)\,dx,\\
J_{22}&=\int \Big[(\na\na_h b \cdot\na) b\cdot \na\na_h u +(\na b \cdot\na)\na_h b\cdot \na\na_h u+(\na_h b \cdot\na)\na b\cdot \na\na_h u\Big]\, dx ,\\
J_{23}&=\int \Big[(\na\na_h b \cdot\na) u\cdot \na\na_h b +(\na b \cdot\na)\na_h u\cdot \na\na_h b+(\na_h b \cdot\na)\na u\cdot \na\na_h b\Big]\, dx.
\end{align*}
It is easy to verify that
\begin{align}\label{J2J4}
\int_0^t ( J_{2}(\tau)+J_{4}(\tau)) d \tau
 \le  C E^{\f 32}(t).
\end{align}
According to (\ref{J11}) and (\ref{IJ11}), 
\begin{align*}
\int_0^t (1+\tau) J_{21}(\tau)d \tau
 \le  C E^{\f 32}(t).
\end{align*}
For $J_{22}$, we further divide it into two parts
\begin{align*}
J_{22}&=\int\big(\na_h^2 b\cdot \na b\cdot\na_h^2u+2\na_h b \cdot \na \na_h b\cdot\na_h^2u\big)\,dx\\
      &+ \int\big(\na_h\p_3 b\cdot \na b\cdot\na_h\p_3u+\p_3 b \cdot \na \na_h b\cdot\na_h\p_3u
        +\na_h b \cdot \na \p_3b\cdot\na_h\p_3u\big)\,dx\\
&:=J_{221}+J_{222}.
\end{align*}
As in (\ref{J12}) and (\ref{IJ12}) for $J_{12}$, 
\begin{align*}
\int_0^t (1+\tau)J_{221}(\tau) d\tau
\le C E^{\f 3 2}(t).
\end{align*}
For $J_{222}$, we have
 \begin{align*}
J_{222}
&\le 3\int |\na_h b|\, |\p_3\na_h b|\,|\p_3\na_h u| \, dx  +2\int |\p_3 b|\, |\na_h^2b| \, |\p_3\na_hu| \, dx\\
      &\quad +\int |\na_h b_3|\,|\p_3^2 b| \, |\p_3\na_h u| \, dx.
\end{align*}
Using the similarities between $J_{222}$ and $J_{131}, J_{132}$ and $J_{133}$, we can easily find
\begin{align*}
\int_0^t (1+\tau)J_{222}(\tau) d\tau
\le C E^{\f 32}(t).
\end{align*}
Therefore, 
\begin{align*}
\int_0^t (1+\tau)J_{22}(\tau) d\tau
\le C E^{\f 32}(t).
\end{align*}
To bound $J_{23}$, we decompose it into  
\begin{align*}
J_{23}
&\le \int(\na_h^2 b\cdot \na u\cdot\na_h^2b+2\na_h b \cdot \na \na_h u\cdot\na_h^2b)\,dx\\
&\quad+ J_{331} + 2J_{332}+J_{333} +J_{334}+ \int |\na_h b_3| \,|\p_3^2 u|\, |\p_3\na_h b|  \, dx.
\end{align*}
The first term and the last term are similar to $J_{32}$ and $J_{335}$, respectively. Thus,
\begin{align*}
\int_0^t (1+\tau)J_{23}(\tau) d\tau
\le C  E^{\f 32}(t).
\end{align*}
Integrating  (\ref{J0}) over $[0,t]$ and invoking (\ref{J1}), (\ref{J3}) and (\ref{J2J4}),  
we derive the desired estimate (\ref{Ghub}). This competes the proof of Lemma \ref{lem42}.
\end{proof}

\vskip .1in
We now turn to the second lemma.
\begin{lemma}\label{lem43}
Assume $(u, b)$ is a solution to (\ref{MHD1}). Then we have
\begin{align}\label{2Ghu}
&-(1+t)(\p_2\na_hu(t), \na_hb(t))+\f 12\int_0^t (1+\tau)\|\p_2\na_h u(\tau)\|_{L^2}^2d\tau\nonumber\\
&\qquad -\f 12\int_0^t (1+\tau)\Big(3\|\p_2\na_hb(\tau)\|_{L^2}^2+\mu^2\|\na_h\p_1^2 u(\tau)\|_{L^2}^2
+\eta^2\|\na_h\Delta_h b(\tau)\|_{L^2}^2\Big)d\tau\nonumber\\
&\quad\le CE(0)+\f 12 E_0(t)+ C E^{\f 32}(t).
\end{align}
\end{lemma}

\begin{proof}[Proof of Lemma \ref{lem43}]
As in (\ref{D2uH2}), we have
\begin{align*}
&-\f {d}{dt}(1+t)(\p_2\na_h u, \na_hb)+ (1+t) \|\p_2\na_h u\|_{L^2}^2- (1+t) \|\p_2\na_h b\|_{L^2}^2\nonumber\\
&=-(\p_2\na_h u, \na_hb) + (1+t)\big(\p_2\na_h(u\cdot\na u), \na_hb\big)-(1+t)\big(\p_2\na_h(b\cdot\na b\big), \na_h b)\nonumber\\
 &\quad +(1+t)\big(\p_2\na_hu, \na_h(u\cdot\na b)\big) -(1+t)\big(\p_2\na_hu, \na_h(b\cdot\na u)\big)\nonumber\\
 &\quad -\mu(1+t)(\p_2\na_h\p_1^2u, \na_hb)-\eta(1+t)(\p_2\na_hu, \na_h\Delta_h b)\nonumber\\
&:= J_5+\cdots+J_{11},
\end{align*}
where $(F,G)$ denotes the $L^2$-inner product of $F$ and $G$. It is clear that
\begin{align*}
&\int_0^t J_{5}(\tau)d\tau \le \f 12 \int_0^t (\|\p_2\na_h u(\tau)\|_{L^2}^2+\|\na_h b(\tau)\|_{L^2}^2)d\tau\le \f 12 E_0(t),\\
&J_{10}=\mu(1+t)(\na_h\p_1^2u, \p_2\na_hb)\le\f {\mu^2} {2}(1+t)\|\p_1^2\na_h u\|_{L^2}^2+\f 12(1+t)\|\p_2\na_h b\|_{L^2}^2,\\
&J_{11}\le\f 12 (1+t)\|\p_2\na_h u\|_{L^2}^2+\f {\eta^2}{2}(1+t)\|\Delta_h\na_h b\|_{L^2}^2.
\end{align*}
Next we bound the nonlinear integral terms. We mainly focus on $J_6$ and $J_8$. The estimates for $J_7$ and $J_9$ can be established similarly.
 By integration by parts and (\ref{Ine1}), we have
\begin{align*}
J_{6}&=-(1+t)\int \na_h u\cdot \na u  \cdot \p_2\na_hb\, dx-(1+t)\int  u\cdot \na\na_h u \cdot\p_2\na_hb\, dx \\
&\le C(1+t)\|\na_h u\|_{L^2}^{\f 12}  \|\p_1\na_h u\|_{L^2}^{\f 12}
        \|\na u\|_{L^2}^{\f 12} \|\p_2\na u\|_{L^2}^{\f 12}
        \|\p_2\na_h b\|_{L^2}^{\f 12} \|\p_3\p_2\na_h b\|_{L^2}^{\f 12}\\
 &  + C(1+t)\|u\|_{L^2}^{\f 12}  \|\p_2 u\|_{L^2}^{\f 12}
        \|\na\na_h u\|_{L^2}^{\f 12} \|\p_1\na\na_h u\|_{L^2}^{\f 12}
        \|\p_2\na_h b\|_{L^2}^{\f 12} \|\p_3\p_2\na_h b\|_{L^2}^{\f 12}.
\end{align*}
Furthermore, 
\begin{align*}
\int_0^t J_{6}(\tau) d\tau
&\le C \sup_{0\le t\le \tau}(1+\tau)^{\f 12}\|\na_h u(\tau)\|_{L^2}^{\f 12} \|\na u(\tau)\|_{L^2}^{\f 12}
       \int_0^t \|\na_h\na u(\tau)\|_{L^2}\\
       &\qquad\times(1+\tau)^{\f 12}\|\p_2\na_h b(\tau)\|_{H^1} d\tau\\
&  + C   \sup_{0\le t\le \tau}(1+\tau)^{\f 12}\|\p_2 u(\tau)\|_{L^2}^{\f 12} \|u(\tau)\|_{L^2}^{\f 12}
       \int_0^t   \|\na_h\na u(\tau)\|_{H^1} \\
       &\qquad\times(1+\tau)^{\f 12}\|\p_2\na_h b(\tau)\|_{H^1} d\tau\\
 & \le C E_2^{\f 14}(t) E_0^{\f 3 4}(t) E_1^{\f 12}(t)\le CE^{\f 32}(t).
\end{align*}
Similarly, we can bound $J_7$ as
\begin{align*}
\int_0^t J_{7}(\tau) d\tau \le CE^{\f 32}(t).
\end{align*}
For $J_8$,  applying the anisotropic inequality (\ref{Ine2}) yields
 \begin{align*}
J_{8}&=(1+t)\int \na_h u\cdot \na b  \cdot\p_2\na_hu\, dx+(1+t)\int  u\cdot \na\na_h b \cdot\p_2\na_hu\, dx \\
&\le C(1+t)\|\na_h u\|_{L^2}^{\f 12}  \|\p_2\na_h u\|_{L^2}^{\f 12}
        \|\na b\|_{L^2}^{\f 14} \|\p_1\na b\|_{L^2}^{\f 14}\|\p_3\na b\|_{L^2}^{\f 14} \|\p_1\p_3\na b\|_{L^2}^{\f 14}
        \|\p_2\na_h u\|_{L^2} \\
 &  \quad+C (1+t)\|u\|_{L^2}^{\f 14} \|\p_2 u\|_{L^2}^{\f 14}\|\p_3u\|_{L^2}^{\f 14} \|\p_2\p_3u\|_{L^2}^{\f 14}
        \|\na\na_h b\|_{L^2}^{\f 12}  \|\p_1\na\na_h b\|_{L^2}^{\f 12}
        \|\p_2\na_h u\|_{L^2}.
\end{align*}
Thus, 
\begin{align*}
\int_0^t J_{8}(\tau) d\tau
&\le C \sup_{0\le t\le \tau}(1+\tau)^{\f 14}\|\na_h u(\tau)\|_{L^2}^{\f 12} \|\na b(\tau)\|_{H^1}^{\f 12}
       \int_0^t (1+\tau)^{\f 34}\|\p_2\na_h u(\tau)\|_{L^2}^{\f 32} \|\p_1\na b(\tau)\|_{H^1}^{\f 12}   d\tau   \\
&  + C   \sup_{0\le t\le \tau}(1+\tau)^{\f 14}\|\na\na_h b(\tau)\|_{L^2}^{\f 12} \|u(\tau)\|_{H^1}^{\f 12}
        \int_0^t  (1+\tau)^{\f 14}\|\p_1\na_h\na b(\tau)\|_{L^2}^{\f 12}\\
        &\qquad\times(1+\tau)^{\f 12}\|\p_2\na_h u(\tau)\|_{L^2}
        \|\p_2 u(\tau)\|_{H^1}^{\f 12} d\tau \\
 & \le C E_1(t) E_0^{\f 1 2}(t) \le CE^{\f 3 2}(t).
\end{align*}
Also,
\begin{align*}
\int_0^t J_{9}(\tau) d\tau \le CE^{\f 3 2}(t).
\end{align*}
Collecting all the estimates for $J_5$ through $J_{11}$, and integrating over $[0,t]$, we derive the desired bound (\ref{2Ghu}). This completes the proof of lemma \ref{lem43}.
\end{proof}

\vskip .1in
We now putting together the two lemmas above to obtain Proposition \ref{pro2}.
\begin{proof}[Proof of Proposition \ref{pro2}]
According to Lemma \ref{lem42} and \ref{lem43}, the combination $(\ref{Ghub})+\lambda_1(\ref{2Ghu})$
yields
\begin{align*}
&(1+t)\big(\|\na_h u(t)\|_{H^1}^2+\|\na_h b(t)\|_{H^1}^2-\lambda_1(\p_2\na_hu, \na_hb)\big)\\
&\qquad +\int_0^t (1+\tau)\Big[\big(2\mu-\f{\mu^2}{2}\lambda_1\big)\|\p_1\na_h u(\tau)\|_{H^1}^2\\
&\qquad\qquad +\f {\lambda_1} {2} (1+\tau)\|\p_2\na_h u(\tau)\|_{L^2}^2
+\big(2\eta-\f 32\lambda_1 - \f {\eta^2}{2}\lambda_1\big)\|\Delta_h b(\tau)\|_{H^1}^2\Big]\, d\tau\nonumber\\
&\quad\le CE(0)+C E_0(t)+C E^{\f 32}(t),
\end{align*}
where $\lambda_1$ is a parameter. If $\lambda_1$ is sufficiently small, then 
\begin{align*}
&(1+t)\Big(\|\na_hu(t)\|_{H^1}^2+\|\na_h b(t)\|_{H^1}^2\Big)\\
&+\int_0^t(1+\tau) \Big(\|\p_1\na_h u(\tau)\|_{H^1}^2 +\|\p_2\na_h u(\tau)\|_{L^2}^2+\|\na_h^2 b(\tau)\|_{H^1}^2\Big)d \tau\nonumber\\
&\le C E(0)+C E_0(t)+ C E^{\f 32}(t).
\end{align*}
This completes the proof of Proposition \ref{pro2}.
\end{proof}

\vskip .3in
\section{Estimate for $E_2(t)$}
\label{sec5}

This section establishes the {\it a priori} inequality \eqref{e2b} for $E_2(t)$. That is, we prove the following proposition.

\begin{proposition}\label{prop3}
Let $(u,b)$ be a  solution to the system (\ref{MHD1}). Then it holds
\begin{align}\label{E2}
E_2(t)&\le  C \Big(E^{\f 32}(t)+ E^2(t)\Big)
+C\Big(\|(u_0,b_0)\|_{H^2}^2+\|(u_0,b_0)\|_{L^2_{x_3}L^1_{x_1x_2}}^2\nonumber\\
&+\|(\p_3u_0,\p_3b_0)\|_{L^2_{x_3}L^1_{x_1x_2}}^2+\|(\p_3^2u_0,\p_3^2b_0)\|_{L^2_{x_3}L^1_{x_1x_2}}^2\Big).
\end{align}
\end{proposition}

We remark that energy estimates are no longer sufficient for the proof of \eqref{E2}.
We resort to the integral representation of (\ref{MHD1}). To convert (\ref{MHD1}) into an integral
representation, we take the Fourier transform of (\ref{MHD1}), solve the linearized system and 
represent the nonlinear system into an integral form via Duhamel's principle. 
The integral representation involves three key kernel functions, which are degenerate 
and anisotropic. Due to the anisotropicity, we divide the frequency space into subdomains 
to obtain sharp upper bounds on the kernel functions. This is done in Proposition \ref{prop4}. 
Once these bounds are at our disposal, we then estimate the $L^2$-norms of $(u, b)$ 
and its derivatives via the integral representation. For the sake of clarity, 
we divide the rest of this section into two subsections.

\vskip .1in
\subsection{Integral representation and bounds for the kernels}

The subsection derives the integral representation of (\ref{MHD1}) and establishes 
optimal upper bounds for the kernel functions. First we recall two basic tools. 
The first one specifies the decay rate of a general heat operator associated with a fractional
Laplacian operator. Here the fractional Laplacian operator can be defined through the Fourier transform
$$
\widehat{\Lambda^\alpha f} (\xi) =|\xi|^{\alpha}\widehat{f}(\xi).
$$
The decay rate is stated in the following lemma, whose proof can be found in many references (see, e.g., \cite{Wu}).

\begin{lemma}\label{lem52}
	Assume $\alpha \ge 0$ and $\beta>0$ are real numbers. Let $1\le p\le q\le\infty$. Then there
	exists a constant $C>0$ such that, for any $t>0$,
	\begin{align*}
    \|\Lambda^\alpha e^{-\Lambda^\beta t}f\|_{L^q(\mathbb R^d)}\le \,C\, t^{-\frac{\alpha}{\beta} - \frac{d}{\beta}(\frac1p-\frac1q)}\, \|f\|_{L^p(\mathbb R^d)}.
	\end{align*}
\end{lemma}

The second tool is an elementary inequality providing upper bounds for a convolution type integral. Its proof is straightforward.

\begin{lemma} \label{lem53}
	Assume $0<s_1\le s_2$. Then, for some constant $C>0$,
	\begin{align}\label{Ine5}
	\int_0^t(1+t-\tau)^{-s_1}(1+\tau)^{-s_2}\ d\tau
	\le\left\{
	\begin{array}{ll}
	C(1+t)^{-s_1},\ \ \ \ \ \text{if}\ \  s_2> 1,
	\\
	C(1+t)^{-s_1} \ln(1+t), \ \text{if}\ \ s_2=1,
	\\
	C(1+t)^{1-s_1-s_2},\ \text{if}\ \ s_2<1.
	\end{array}
	\right.
	\end{align}
\end{lemma}

\vskip .1in
Now we derive an integral representation of (\ref{MHD1}). Applying the Leray-Hopf projection operator $\mathbb P=I-\na\Delta^{-1}\na\cdot$ to the velocity equation in (\ref{MHD1}) and  taking the Fourier transform of the resulting equations, we have
\begin{equation}\label{pp}
\p_t\bordermatrix{\text{}\cr
	&\widehat{u} \cr
	&\widehat{b} }=A\bordermatrix{\text{}\cr
	&\widehat{u} \cr
	&\widehat{b} }+\bordermatrix{\text{}\cr
	&\widehat{N}_1 \cr
	&\widehat{N}_2 },
\end{equation}
where
$$
A=\bordermatrix{\text{}\cr
	& -\mu\xi_1^2 &  i\xi_2 \cr
	& i\xi_2      &  -\eta|\xi_h|^{2} }, \qquad
N_1= \mathbb P (b\cdot\nabla b-u\cdot\nabla u), \qquad
N_2=b\cdot\nabla u-u\cdot\nabla b
$$
with $|\xi_h|^{2}=\xi_1^{2}+\xi_2^{2}$.
To diagonalize $A$, we compute the eigenvalues of $A$, 
$$
\lambda_1=\f{-(\mu\xi_1^2+\eta |\xi_h|^2)-\sqrt{\Gamma }}{2},\qquad  \lambda_2=\f{-(\mu\xi_1^2+\eta |\xi_h|^2)+\sqrt{\Gamma }}{2},
$$
where
$$
\Gamma=(\mu\xi_1^2+\eta |\xi_h|^2)^2-4(\mu\eta\xi_1^2|\xi_h|^2+\xi_2^2).
$$
The corresponding  eigenvectors are
$$
\rho_1=\bordermatrix{\text{}\cr
	&\lambda_1+\eta|\xi_h|^2\cr
	& i\xi_2}, \qquad
\rho_2=\bordermatrix{\text{}\cr
	&\lambda_2+\eta|\xi_h|^2\cr
	& i\xi_2}.
$$
Therefore, the matrix $A$ can be diagonalized as
\begin{align} \label{pp1}
A=(\rho_1,\rho_2 )
\bordermatrix{\text{}\cr
	& \lambda_1 & 0 \cr
	& 0 & \lambda_2 } \
(\rho_1,\rho_2 )^{-1}.
\end{align}
We can now represent \eqref{pp} as
$$
\bordermatrix{\text{}\cr
	&\widehat{u} \cr
	&\widehat{b} }
=e^{At}\ \bordermatrix{\text{}\cr
	&\widehat{u_0} \cr
	&\widehat{b_0} }
+\int_0^t e^{A(t-\tau)}\ \bordermatrix{\text{}\cr
	&\widehat{N_1}(\tau) \cr
	&\widehat{N_2}(\tau) }d\tau.
$$
By making $e^{At}$ more explicit via (\ref{pp1}), we obtain the integral representation
\ben
&& \hskip -.4in\widehat{u}(\xi, t) = \widehat{Q}_1(t)\widehat{u}_0+\widehat{Q}_2(t)\widehat{b}_0+\int_0^t\big(\widehat{Q}_1(t-\tau)\widehat{N}_1(\tau)
+\widehat{Q}_2(t-\tau)\widehat{N}_2(\tau)\big)\,d\tau, \label{uin}\\
&& \hskip -.4in  \widehat{b}(\xi, t) = \widehat{Q}_2(t)\widehat{u}_0+\widehat{Q}_3(t)\widehat{b}_0+\int_0^t\big(\widehat{Q}_2(t-\tau)\widehat{N}_1(\tau)
+\widehat{Q}_3(t-\tau)\widehat{N}_2(\tau)\big)\,d\tau, \label{bin}
\een
where
\begin{align}
 \widehat{Q}_1(t) = \eta|\xi_h|^2G_1(t)+G_2(t), \quad \widehat{Q}_2(t) = i\xi_2\, G_1(t), \quad
 \widehat{Q}_3(t) = -\eta|\xi_h|^2G_1(t)+G_3(t), \label{Ms}
\end{align}
with
\beno
&& G_1(t)=\frac{e^{\lambda_2t}-e^{\lambda_1t}}{\lambda_2-\lambda_1},\qquad G_2(t)=\frac{\lambda_2e^{\lambda_2t}-\lambda_1e^{\lambda_1t}}{\lambda_2-\lambda_1}=e^{\lambda_2t}+\lambda_1G_1(t), \notag\\
&& G_3(t)=\frac{\lambda_2e^{\lambda_1t}-\lambda_1e^{\lambda_2t}}{\lambda_2-\lambda_1}=e^{\lambda_1t}-\lambda_1G_1(t).
\eeno
We remark that when $\lambda_1 =\lambda_2$, the representation in (\ref{uin}) and (\ref{bin}) remains valid
if we replace $G_1$ by its limiting form
$$
G_1(t) =\lim_{\lambda_2 \to \lambda_1} \frac{e^{\lambda_2t}-e^{\lambda_1t}}{\lambda_2-\lambda_1} = t \,e^{\lambda_1 t}.
$$

\vskip .1in
Next we investigate the behaviors of the kernels $\widehat Q_i(\xi,t)\ (i=1,2,3)$,  which play a crucial role in the estimate of $E_2(t)$. There kernels are anisotropic and degenerate. To obtain precise and sharp upper bounds,  we  divide the frequency space
into subdomains  and classify the behavior of the kernel functions in each subdomain.

\begin{proposition}\label{prop4}
	The domain $\R^3$ is split into two subdomains,  $\R^3 = A_1 \cup A_2$ with
	\begin{align*}
	&A_1:=\left\{ \xi\in\R^3:  \sqrt{\Gamma}\le\f{\mu \xi_1^2+\eta |\xi_h|^2}{2} \ \ \text {or}\ \
	3(\mu \xi_1^2+\eta |\xi_h|^2)^2\le 16(\mu\eta\xi_1^2|\xi_h|^2+\xi_2^2)\right \},\\
	&A_2:=\left\{ \xi\in\R^3:  \sqrt{\Gamma}>\f{\mu \xi_1^2+\eta |\xi_h|^2}{2} \ \ \text {or} \ \
	3(\mu \xi_1^2+\eta |\xi_h|^2)^2 > 16(\mu\eta\xi_1^2|\xi_h|^2+\xi_2^2)\right \}.
	\end{align*}
	Then we have
	\begin{enumerate}
	\item[(1)] There exist two constants $C>0$ and $c_0>0$ such that,	for any $\xi\in A_1$,
	\begin{align*}
	& Re\lambda_1\le -\f{\mu \xi_1^2+\eta |\xi_h|^2}{2},\ Re\lambda_2\le -\f{\mu \xi_1^2+\eta |\xi_h|^2}{4},\\
	& |G_1(t)|\le t e^{-\f{\mu \xi_1^2+\eta |\xi_h|^2}{4}t},\ \quad \  |\widehat{Q}_i(\xi, t)|\le Ce^{-c_0|\xi_h|^2t},\ i=1, 2, 3.
	\end{align*}
	
   \item[(2)] There is a constant $C>0$ such that, for any $\xi\in A_2$,
	\begin{align*}
    & \lambda_1\ < -\f{3(\mu \xi_1^2+\eta |\xi_h|^2)}{4},\ \lambda_2\le -\f{\mu\eta\xi_1^2|\xi_h|^2+\xi_2^2}{\mu \xi_1^2+\eta |\xi_h|^2},\\
    & |G_1(t)| <\f{2}{\mu \xi_1^2+\eta |\xi_h|^2} \left(e^{-\f 34 (\mu \xi_1^2+\eta |\xi_h|^2)t}+e^{-\f{\mu\eta\xi_1^2|\xi_h|^2+\xi_2^2}{\mu \xi_1^2+\eta |\xi_h|^2}t}\right),\\
    & |\widehat{Q}_i(t)| < C(e^{-\f 34 (\mu \xi_1^2+\eta |\xi_h|^2)t}+e^{-\f{\mu\eta\xi_1^2|\xi_h|^2+\xi_2^2}{\mu \xi_1^2+\eta |\xi_h|^2}t}),\ i=1, 2, 3.
	\end{align*}
	If we further divide $A_2$ into three subdomains $A_{21}, A_{22}, A_{23}$,
	\beno
	&& A_{21}=\{\xi\in A_2, \quad |\xi_1|\sim| \xi_2|\},\\
	&& A_{22}=\{\xi\in A_2, \quad |\xi_1|>>|\xi_2|\},\\
	&& A_{23}=\{\xi\in A_2, \quad |\xi_1|<<|\xi_2|\},
	\eeno
	then, for some constants $C>0, c_1>0, c_2>0, c_3>0$ and $i=1,2,3$,
	\begin{align*}
    & |\widehat{Q}_i(t)| \le C\, e^{-c_1  |\xi_h|^2 t}, \quad \mbox{if}\quad  \xi \in A_{21}, \\
    & |\widehat{Q}_i(t)| \le C\, e^{-c_1 |\xi_h|^2 t}, \quad \mbox{if}\quad  \xi \in A_{22}, \\
    & |\widehat{Q}_i(t)| \le C\,( e^{-c_1 |\xi_h|^2 t}+ e^{-c_2 \xi_1^2 t-c_3t}), \quad \mbox{if}\quad  \xi \in A_{23}.
   \end{align*}
	\end{enumerate}
\end{proposition}

\begin{proof}[Proof of Proposition \ref{prop4}]
	$(1)$  For $\xi\in A_1$,  $\sqrt{\Gamma}\le \f{\mu\xi_1^2+\eta|\xi_h|^2}{2}$.  Through the direct estimates and the mean-value theorem, we have
	\begin{align}
	&-\f{3(\mu\xi_1^2+\eta|\xi_h|^2)}{4}\le Re\lambda_1\le-\f{\mu\xi_1^2+\eta|\xi_h|^2}{2}, \notag\\
    &Re\lambda_2\le -\f{\mu\xi_1^2+\eta|\xi_h|^2}{4},\,\quad
     |G_1(t)|\le t e^{-\f{\mu\xi_1^2+\eta|\xi_h|^2}{4}t}.\label{yy}
	\end{align}
    To bound the kernel functions $\widehat{Q}_1(t)$ and $\widehat{Q}_3(t)$, we consider two cases: $\lambda_1$ is a real number and  $\lambda_1$ is an imaginary number. If $\lambda_1$ is a real number,  for some pure constant $c_0$ dependent of $\mu$ and $\eta$, we have
	\begin{align*}
    |\widehat{Q}_1(t)|&=\Big|\eta |\xi_h|^2 G_1(t)+\lambda_1G_1(t)+e^{\lambda_2t}\Big|\\
    &\le C (\mu\xi_1^2+\eta|\xi_h|^2)te^{-\f{\mu\xi_1^2+\eta|\xi_h|^2}{4}t}+e^{-\f{\mu\xi_1^2+\eta|\xi_h|^2}{4}t}\\
    &\le  Ce^{-c_0|\xi_h|^2t},
	\end{align*}
	where we have used the simple fact that
    $x\ e^{-x}\le C$ for $x\ge 0$. If $\lambda_1$ is an imaginary number, namely $\Gamma<0$, then
   $$|\lambda_1|^2=\mu\eta\xi_1^2|\xi_h|^2+\xi_2^2,\quad\Gamma=4|\lambda_1|^2-(\mu\xi_1^2+\eta|\xi_h|^2)^2. $$
   Clearly, (\ref{yy}) implies
   $$
   \Big|\eta |\xi_h|^2 G_1(t)+e^{\lambda_2t}\Big|\le Ce^{-c_0|\xi_h|^2t}.
   $$
   Now we bound $|\lambda_1G_1(t)|$.
   we further divide the consideration into two subcases: $|\lambda_1|\le | \sqrt{\Gamma}\ |$ and $ |\lambda_1| \geq |\sqrt{\Gamma}\ |$  .
   In the case when $|\lambda_1|\le |\sqrt{\Gamma}|$, by the definition of $G_1$, we obtain
    \begin{align*}
	|\lambda_1G_1(t)|=\f{|\lambda_1|}{|\sqrt{\Gamma}|}\,|e^{\lambda_1t}-e^{\lambda_2t}|\le Ce^{-\f{\mu\xi_1^2+\eta|\xi_h|^2}{4}t}.
    \end{align*}
   In the case when $|\lambda_1| \geq |\sqrt{\Gamma}\ |=\sqrt{-\Gamma}$, we have
    \begin{align*}
    |\lambda_1|^2\ge 4 |\lambda_1|^2-(\mu\xi_1^2+\eta|\xi_h|^2)^2,
    \end{align*}
   or
    \begin{align*}
     \sqrt{3} |\lambda_1|\le \mu\xi_1^2+\eta|\xi_h|^2.
    \end{align*}
   Thus,
    \begin{align*}
    |\lambda_1G_1(t)| \le \f 1 {\sqrt{3}}(\mu\xi_1^2+\eta|\xi_h|^2)|G_1| \le C(\mu\xi_1^2+\eta|\xi_h|^2) te^{-\f{(\mu\xi_1^2+\eta|\xi_h|^2)}{4}t}\le  Ce^{-c_0|\xi_h|^2t}.
    \end{align*}
  Consequently, if $\lambda_1$ is an imaginary number, we derive
    \begin{align*}
	 |\widehat{Q}_1(t)|\le Ce^{-c_0|\xi_h|^2t}.
	\end{align*}
In summary, for $\xi\in A_1$,
     \begin{align*}
	 |\widehat{Q}_1(t)|\le Ce^{-c_0|\xi_h|^2t}.
	\end{align*}
Similarly, we have
	\begin{align*}
	|\widehat{Q}_3(t)|=\Big|-\eta |\xi_h|^2 G_1(t)-\lambda_1G_1(t)+e^{\lambda_1t}\Big|
	\le Ce^{-c_0|\xi_h|^2t}.
	\end{align*}

   Now we bound $\widehat{Q}_2(t)$.  As in the estimate of $\widehat{Q}_1(t)$,
    we consider the following two cases:
    $|\xi_2|\le { |\sqrt{\Gamma}|}$ and $|\xi_2|\geq { |\sqrt{\Gamma}|}$.
   In the first case $|\xi_2|\le |\sqrt{\Gamma}\ |$, by the definition of $\widehat{Q}_2(t)$ in (\ref{Ms}),
	$$
	|\widehat{Q}_2(t)|=\Big|\f{\xi_2}{\sqrt{\Gamma}}\Big|\,|e^{\lambda_1t}-e^{\lambda_2t}| \le C\,e^{-c_0|\xi_h|^2t}.
	$$
	 In the second case, \,$|\xi_2|\geq |\sqrt{\Gamma}\ |$,
	$$
	\big|(\mu\xi_1^2+\eta|\xi_h|^2)^2-4(\mu \eta \xi_1^2|\xi_h|^2+\xi_2^2)\big|\le \xi_2^2,
	$$
	or
	$$
	-\xi_2^2 \le (\mu\xi_1^2+\eta|\xi_h|^2)^2-4(\mu \eta \xi_1^2|\xi_h|^2+\xi_2^2)\le \xi_2^2,
	$$
	which implies
	$$
	3\xi_2^2+4\mu \eta \xi_1^2|\xi_h|^2\le (\mu\xi_1^2+\eta|\xi_h|^2)^2.
	$$
	In particular,
    $$
	 \sqrt{3} |\xi_2|\le \mu\xi_1^2+\eta|\xi_h|^2.
	$$
    Therefore,
	$$
	|\widehat{Q}_2(t)|\le \ \f {1}{\sqrt{3}}(\mu\xi_1^2+\eta|\xi_h|^2) \ |G_1(t)|
    \le C\,e^{-c_0|\xi_h|^2t}.
	$$
	
	$(2)$  For $\xi \in A_2$, we have $\f {\mu\xi_1^2+\eta|\xi_h|^2} {2}<\sqrt{\Gamma}\le \mu\xi_1^2+\eta|\xi_h|^2$. It then follows that
	\begin{align*}
	&-(\mu\xi_1^2+\eta|\xi_h|^2)\le\lambda_1< -\f 34 (\mu\xi_1^2+\eta|\xi_h|^2),\\
	&\lambda_2=\f{\Gamma-(\mu\xi_1^2+\eta|\xi_h|^2)^2}{2(\mu\xi_1^2+\eta|\xi_h|^2+\sqrt{\Gamma})}
    \le-\f{\mu \eta \xi_1^2|\xi_h|^2+\xi_2^2}{\mu\xi_1^2+\eta|\xi_h|^2}.
	\end{align*}
	Therefore,
    $$
     |G_1(t)|\le\f {1}{\lambda_2-\lambda_1}(e^{\lambda_1t}+e^{\lambda_2t})<\f {2} {\mu\xi_1^2+\eta|\xi_h|^2}\Big(e^{-\f34 (\mu\xi_1^2+\eta|\xi_h|^2)t}+e^{-\f{\mu \eta \xi_1^2|\xi_h|^2+\xi_2^2}{\mu\xi_1^2+\eta|\xi_h|^2}t}\Big).
    $$
    As a consequence,
	\begin{align*}
	|\widehat{Q}_1(t)|&=\Big|\eta |\xi_h|^2 G_1(t)+\lambda_1G_1(t)+e^{\lambda_2t}\Big|\\
	&\le 2(\mu\xi_1^2+\eta|\xi_h|^2)|G_1(t)|+e^{\lambda_2t}\\
    &< C\Big(e^{-\f34 (\mu\xi_1^2+\eta|\xi_h|^2)t}
    +e^{-\f{\mu \eta \xi_1^2|\xi_h|^2+\xi_2^2}{\mu\xi_1^2+\eta|\xi_h|^2}t}\Big).
	\end{align*}
    Similarly,
    \begin{align*}
	|\widehat{Q}_3(t)|=\Big|-\eta |\xi_h|^2 G_1(t)-\lambda_1G_1(t)+e^{\lambda_1t}\Big|
	< C\Big(e^{-\f34 (\mu\xi_1^2+\eta|\xi_h|^2)t}+e^{-\f{\mu \eta \xi_1^2|\xi_h|^2+\xi_2^2}{\mu\xi_1^2+\eta|\xi_h|^2}t}\Big).
	\end{align*}
	Due to $\sqrt{\Gamma}> \f{\mu\xi_1^2+\eta|\xi_h|^2}{2}$, we find
	\begin{align*}
	\f{3}{4}(\mu\xi_1^2+\eta|\xi_h|^2)^2>4(\mu \eta \xi_1^2|\xi_h|^2+\xi_2^2)\geq \xi_2^2.
	\end{align*}
	Therefore,
	\begin{align*}
    |\widehat{Q}_2(t)|< C (\mu\xi_1^2+\eta|\xi_h|^2)\ |G_1(t)|< C(e^{-\f34 (\mu\xi_1^2+\eta|\xi_h|^2)t}
    +e^{-\f{\mu \eta \xi_1^2|\xi_h|^2+\xi_2^2}{\mu\xi_1^2+\eta|\xi_h|^2}t}\Big).
	\end{align*}
	
   Finally, according to the upper bound for  $|\widehat{Q}_i(t)| \ (i=1,2,3)$, by further division of $A_2$ into $A_{21}$, $A_{22}$ and $A_{23}$, we can establish more definite upper bound.  For
   $\xi \in A_{21}$, $\xi_1^2 \sim \xi_2^2$ , we have
   $$
   \f{\mu\eta\xi_1^2|\xi_h|^2+\xi_2^2}{\mu \xi_1^2+\eta |\xi_h|^2} \quad\sim \quad |\xi_h|^2+1.
   $$
   For
   $\xi \in A_{22}$, $\xi_1^2 >> \xi_2^2$, there exists a $c_1>0$ small sufficiently such that
   $$
  \f{\mu\eta\xi_1^2|\xi_h|^2+\xi_2^2}{\mu \xi_1^2+\eta |\xi_h|^2}\ge c_1 |\xi_h|^2.
   $$
 The behavior $\xi\in A_{23}$ can be similarly identified. This completes the proof of Proposition \ref{prop4}.
\end{proof}

\vskip .1in
\subsection{Proof of Proposition \ref{prop3} }
With these preparations at our disposal, we are now ready to prove Proposition \ref{prop3}. Since the process is complicated and long, the proof is divided into three lemmas. To do so, we make the following decomposition for $E_2(t)$,
$$
E_2(t)=E_{21}(t) +E_{22}(t) +E_{23}(t),
$$
where
\begin{align*}
&E_{21}(t)=\sup_{0\le\tau\le t}
(1+\tau)\|(u(\tau),b(\tau))\|_{L^2}^2,\\
&E_{22}(t)=\sup_{0\le\tau\le t}(1+\tau)^2\|(\na_h u(\tau),\na_h b(\tau))\|_{L^2}^2+\sup_{0\le\tau\le t}(1+\tau)^{1-2\varepsilon }\|(\p_3u(\tau),\p_3b(\tau))\|_{L^2}^2,\nonumber\\
&E_{23}(t)=\sup_{0\le\tau\le t}\sum_{k=1}^{2}(1+\tau)^{\f 52-2\varepsilon}\|(\p_1\p_k u(\tau), \p_1\p_k b(\tau))\|_{L^2}^2\nonumber\\
&\qquad\qquad 
+\sup_{0\le\tau\le t}(1+\tau)^{2-2\varepsilon }\|(\p_1\p_3 u(\tau), \p_1\p_3 b(\tau))\|_{L^2}^2\\
&\qquad\qquad +\sup_{0\le\tau\le t}\sum_{k=2}^{3}(1+\tau)^{\f 43- 2\varepsilon}\|(\p_2\p_k u(\tau), \p_2\p_k b(\tau))\|_{L^2}^2\\
&\qquad\qquad +\sup_{0\le\tau\le t}(1+\tau)^{\f 12 }\|(\p_3^2 u(\tau), \p_3^2 b(\tau))\|_{L^2}^2.
\end{align*}

Without loss of generality, we assume
$t>1$. In fact, if $0\le t\le 1$,  Proposition \ref{pro1} implies
\begin{align}\label{E20}
E_2(t)\le C \sup_{0\le\tau\le t}\|(u(\tau),b(\tau))\|_{H^2}^2\le C E_0(t)\le CE(0)+ CE^{\f 32}(t).
\end{align}
Next we present the estimates for $E_{21}(t)$, $E_{22}(t)$ and $E_{23}(t)$, which will be shown in three lemmas. Proposition \ref{prop3} then follows as an immediate consequence. 

\begin{lemma} \label{lem55}
Assume that $(u,b)$ is a solution  to (\ref{MHD1}). Then we have
\begin{align}\label{E21}
E_{21}(t)\le C E^2(t)+C(\|(u_0,b_0)\|_{L_{x_3}^2L_{x_1x_2}^1}^2+\|(u_0,b_0)\|_{L^2}^2).
\end{align}
\end{lemma}

\begin{proof}[Proof of Lemma \ref{lem55}]
Recalling (\ref{uin}) and (\ref{bin}), and applying Plancherel' theorem, we have
\begin{align}
	\|u(t)\|_{L^2(\R^3)}&=\|\widehat{u}( t)\|_{L^2(\R^3)}
	\le\|\widehat{Q}_1(t)\widehat{u}_0\|_{L^2(\R^3)}+\|\widehat{Q}_2(t)\widehat{b}_0\|_{L^2(\R^3)}\nonumber\\
	&+\int_0^t\|\widehat{Q}_1(t-\tau)\widehat{N}_1(\tau)\|_{L^2(\R^3)}\,d\tau
	+\int_0^t\|\widehat{Q}_2(t-\tau)\widehat{N}_2(\tau)\|_{L^2(\R^3)}\,d\tau.\label{uL2},\\
    \|b(t)\|_{L^2(\R^3)}&=\|\widehat{b}( t)\|_{L^2(\R^3)}
	\le\|\widehat{Q}_2(t)\widehat{u}_0\|_{L^2(\R^3)}+\|\widehat{Q}_3(t)\widehat{b}_0\|_{L^2(\R^3)}\nonumber\\
	&+\int_0^t\|\widehat{Q}_2(t-\tau)\widehat{N}_1(\tau)\|_{L^2(\R^3)}\,d\tau
	+\int_0^t\|\widehat{Q}_3(t-\tau)\widehat{N}_2(\tau)\|_{L^2(\R^3)}\,d\tau.\label{bL2}
\end{align}
We shall only provide the estimates for $\|u\|_{L^2(\R^3)}$. $\|b\|_{L^2(\R^3)}$ can be 
estimated in a similar way and admits the same bound as $u$ due to the similarity of (\ref{bL2})
 with (\ref{uL2}) . We  focus on the first term and the third term on the right side 
 in (\ref{uL2}). The estimates for the rest can be established similarly.
By Proposition \ref{prop4} and Lemma {\ref{lem52}},
\begin{align}\label{Bu0}
	\|\widehat{Q}_1(t)\widehat{u}_0\|_{L^2{(\R^3)}}
	&\le C \|e^{-{\widetilde{c}_0}|\xi_h|^2t}\,\widehat{u}_0 \ \|_{L^2{(\R^3)}}
	+C\|e^{ -{ c_3}t } \ \widehat{ u}_0\|_{L^2(\R^3)}\nonumber\\
	&= C \big\| \ \|e^{-{\widetilde{c}_0}(\Lambda_1^2+\Lambda_2^2)\,t}\, u_0 \|_{L^2_{x_1x_2}}\ \big\|_{L_{x_3}^2}
	+C\|e^{ -{ c_3}t } \ \widehat{ u}_0\|_{L^2}\nonumber\\
	&\le C(1+t)^{-\f{1}{2}}(\|u_0\|_{L_{x_3}^2L_{x_1x_2}^1}+\|u_0\|_{L^2}),
	\end{align}
where we have used the fact  $e^{-{ c_3}t}(1+t)^{m}\le C({ c_3},m)$ for any $m\ge 0$.
For the third term, according to the upper bound for $\widehat{Q}_1(t)$,
\begin{align}\label{Iu3}
	&\int_0^t\|\widehat{Q}_1(t-\tau)\widehat{N}_1(\tau)\|_{L^2{(\R^3)}}\,d\tau
	\le\int_0^t\|\widehat{Q}_1(t-\tau)\widehat{M}_1(\tau)\|_{L^2{(\R^3)}}\,d\tau\nonumber\\
	&\quad\le C\int_0^t \|e^{-{\widetilde{ c}_0}|\xi_h|^2(t-\tau)}\widehat{M}_1(\tau)\|_{L^2{(\R^3)}}\,d\tau
	+C\int_0^t e^{-{ c_3}(t-\tau)}\|\widehat{M}_1(\tau)\|_{L^2{(\R^3)}}\,d\tau\nonumber\\
    &\quad= C\int_0^{t-1} \|e^{-{\widetilde{ c}_0}|\xi_h|^2(t-\tau)}\widehat{M}_1(\tau)\|_{L^2{(\R^3)}}\,d\tau
	+ C\int_{t-1}^{t} \|e^{-{\widetilde{ c}_0}|\xi_h|^2(t-\tau)}\widehat{M}_1(\tau)\|_{L^2{(\R^3)}}\,d\tau\nonumber\\
	&\qquad+C\int_0^t e^{-{ c_3}(t-\tau)}\|M_1(\tau)\|_{L^2{(\R^3)}}\,d\tau,
\end{align}
where $ M_1=b\cdot \na b-u\cdot \na u$ and we have used the fact that the projection operator $\mathbb{P}$ is bounded in $L^2$.
Observing the simple facts, for any positive number  $m$, 
$$
 (1+t-\tau)^{-m}\ge 2^{-m}\,\, \text{for}\,\, {\tau\in[t-1,t]}\quad\mbox{and}\quad   e^{-{ c_3}t}(1+t)^{m}\le C({ c_3},m)\,\,\text{for}\, t>0,
$$
we have 
$$
\int_{t-1}^{t} \|e^{-{\widetilde{ c}_0}|\xi_h|^2(t-\tau)}\widehat{M}_1(\tau)\|_{L^2{(\R^3)}}\,d\tau
\le 2^m \int_{t-1}^{t} (1+t-\tau)^{-m} \|\widehat{M}_1(\tau)\|_{L^2{(\R^3)}}\,d\tau.
$$
Then (\ref{Iu3}) can be further bounded as
\begin{align}\label{Iu31}
	&\int_0^t\|\widehat{Q}_1(t-\tau)\widehat{N}_1(\tau)\|_{L^2{(\R^3)}}\,d\tau\nonumber\\
	&\le C\int_0^{t-1} \|e^{-{\widetilde{ c}_0}|\xi_h|^2(t-\tau)}\widehat{M}_1(\tau)\|_{L^2{(\R^3)}}\,d\tau
	+ C\int_{0}^{t} (1+t-\tau)^{-m}\|\widehat{M}_1(\tau)\|_{L^2{(\R^3)}}\,d\tau.
\end{align}
Next we bound the terms on the right side in (\ref{Iu31}). It suffices to estimate the integral involving $u\cdot\na u$. The integral of $b\cdot\na b$ admits the same bound.
As in (\ref{Bu0}), we have
\begin{align*}
	\int_0^{t-1} \|e^{-{\widetilde{ c}_0}|\xi_h|^2(t-\tau)} \widehat {u\cdot \na u}(\tau) \|_{L^2{(\R^3)}}\,d\tau
    \le C \int_0^{t} (1+t-\tau)^{-\f 12}  \| u\cdot \na u(\tau)\|_{L_{x_3}^2L_{x_1x_2}^1} \,d\tau.
\end{align*}
By (\ref{fg}), 
\begin{align}\label{Eu}
\| u\cdot \na u\|_{L_{x_3}^2L_{x_1x_2}^1}
\le  C\|u_h\|_{L^2}^{\f 12} \|\p_3u_h\|_{L^2}^{\f 12}  \|\na_h u\|_{L^2}
+ C \|u_3\|_{L^2}^{\f 12} \|\p_3u_3\|_{L^2}^{\f 12} \|\p_3u\|_{L^2}.
\end{align}
By Lemma {\ref{lem53}}, 
\begin{align}\label{Eut}
	&\int_0^{t-1} \|e^{-{\widetilde{ c}_0}|\xi_h|^2(t-\tau)} \widehat {u\cdot \na u}(\tau) \|_{L^2{(\R^3)}}\,d\tau\nonumber\\
&\le C \sup_{0\le \tau \le t} (1+\tau)^{\f 14}\|u_h(\tau)\|_{L^2}^{\f 12}  (1+\tau)^{\f 1 4-\f 12 \varepsilon}\|\p_3u_h(\tau)\|_{L^2}^{\f 12}
     (1+\tau)\|\na_h u(\tau)\|_{L^2}\nonumber\\
    &\quad\times \int_0^{t} (1+t-\tau)^{-\f 12} (1+\tau)^{-\f 3 2+\f 12 \varepsilon}\,d\tau\nonumber\\
 & + C \sup_{0\le \tau \le t} (1+\tau)^{\f 14}\|u_3(\tau)\|_{L^2}^{\f 12}  (1+\tau)^{\f 1 2}\|\p_3u_3(\tau)\|_{L^2}^{\f 12}
     (1+\tau)^{\f 1 2- \varepsilon}\|\p_3 u(\tau)\|_{L^2}\nonumber\\
     &\quad\times \int_0^{t} (1+t-\tau)^{-\f 12} (1+\tau)^{-\f 5 4+ \varepsilon}\,d\tau\nonumber\\
&\le C E_2(t) (1+t)^{-\f 12}.
\end{align}
Applying H\"{o}lder's inequality and Sobolev's inequality, the second integral involving $u\cdot \na u$ in (\ref{Iu31}) can be bounded as
\begin{align*}
\int_{0}^{t}& (1+t-\tau)^{-m}\|u\cdot \na u(\tau)\|_{L^2{(\R^3)}}\,d\tau
\le C \int_0^t (1+t-\tau)^{-m}\|u(\tau)\|_{L^4}\|\na u(\tau)\|_{L^4}\,d\tau \\
&\le C \int_0^t (1+t-\tau)^{-m}\|u(\tau)\|_{L^2}^{\f 14}\|\na u(\tau)\|_{L^2}\|\na^2 u(\tau)\|_{L^2}^{\f 34}\,d\tau \\
&\le CE_2(t)\int_0^{t} (1+t-\tau)^{-m} (1+\tau)^{-\f {13} {16}+ \varepsilon}\,d\tau
\le  CE_2(t)(1+t)^{-\f 12},
\end{align*}
where $m>1$. As a consequence, we have
   \begin{align*}
	\int_0^t\|\widehat{Q}_1(t-\tau)\widehat {u\cdot \na u}(\tau)\|_{L^2{(\R^3)}}\,d\tau
	\le C(1+t)^{-\f 12} E_2(t).
  \end{align*}
Thereby, we infer
   \begin{align}\label{Bu3}
	\int_0^t\|\widehat{Q}_1(t-\tau)\widehat{N}_1(\tau)\|_{L^2{(\R^3)}}\,d\tau
	\le C(1+t)^{-\f 12} E_2(t).
   \end{align}
The second term and the fourth term admit the similar bound 
as $(\ref{Bu0})$  and (\ref{Bu3}), respectively. Therefore, we can conclude
   \begin{align*}
	(1+t)^{\f 12}\|u(t)\|_{L^2}
	\le C\Big(E(t)+\|(u_0,b_0)\|_{L_{x_3}^2L_{x_1x_2}^1}+\|(u_0,b_0)\|_{L^2}\Big),
   \end{align*}
which means
   \begin{align*}
	(1+t)\|u(t)\|_{L^2}^2
	\le C\Big(E^2(t)+\|(u_0,b_0)\|_{L_{x_3}^2L_{x_1x_2}^1}^2+\|(u_0,b_0)\|_{L^2}^2\Big).
   \end{align*}
Also, $\|b\|_{L^2}$ obeys the same bound. This complete the proof of Lemma \ref{lem55}.
\end{proof}

\vskip .1in
\begin{lemma} \label{lem56}
Let $(u,b)$ be a solution  to (\ref{MHD1}). Then we have
\begin{align}\label{E22}
E_{22}(t)&\le C E^2(t)
+C\big(\|(u_0,b_0)\|_{L_{x_3}^2L_{x_1x_2}^1}^2+\|(\p_3u_0,\p_3b_0)\|_{L_{x_3}^2L_{x_1x_2}^1}^2+\|(\na u_0,\na b_0)\|_{L^2}^2\big).
\end{align}
\end{lemma}

\vskip .1in
\begin{proof}[Proof of Lemma \ref{lem56}]
By differentiating (\ref{uin}) and (\ref{bin}), we have, for $i=1,2,3$,
	\ben
    \widehat{\p_i u}(\xi, t) &=& \widehat{Q}_1(t)\widehat{\p_i u}_0+\widehat{Q}_2(t)\widehat{\p_i b}_0\notag\nonumber\\
	&& +\int_0^t\big(\widehat{Q}_1(t-\tau)\widehat{\p_i N}_1(\tau)
	+\widehat{Q}_2(t-\tau)\widehat{\p_i N}_2(\tau)\big)\,d\tau, \nonumber\\
	\widehat{\p_i b}(\xi, t) &=& \widehat{Q}_2(t)\widehat{\p_i u}_0+\widehat{Q}_3(t)\widehat{\p_i b}_0 \notag \nonumber\\
	&& +\int_0^t\big(\widehat{Q}_2(t-\tau)\widehat{\p_i N}_1(\tau)
	+\widehat{Q}_3(t-\tau)\widehat{\p_i N}_2(\tau)\big)\,d\tau. \nonumber
	\een
As in the proof of Lemma \ref{lem55}, we focus on the $\|\p_i u(t)\|_{L^2}$. Clearly,
 \begin{align}
 \|\p_iu(t)\|_{L^2(\R^3)}&=\|\widehat{\p_iu}( t)\|_{L^2(\R^3)}
	\le\|\widehat{Q}_1(t)\widehat{\p_iu}_0\|_{L^2(\R^3)}+\|\widehat{Q}_2(t)\widehat{\p_ib}_0\|_{L^2(\R^3)}\nonumber\\
	&+\int_0^t\|\widehat{Q}_1(t-\tau)\widehat{\p_iN}_1(\tau)\|_{L^2(\R^3)}\,d\tau
	+\int_0^t\|\widehat{Q}_2(t-\tau)\widehat{\p_iN}_2(\tau)\|_{L^2(\R^3)}\,d\tau\nonumber\\
 &:=H_{i1}+H_{i2}+H_{i3}+H_{i4}.\label{iuL2}
 \end{align}
It suffices to bound $H_{i1}$ and $H_{i3}$ in (\ref{iuL2}). $H_{i2}$ and $H_{i4}$ share similar estimates as $H_{i1}$ and $H_{i3}$, respectively.

  \vskip .1in
  ${\bf(1)}$ \,${\bf i=1}$ or ${\bf i=2}$ .
  \vskip .1in
 We focus on the case $i=2$. The case $i=1$ is similar. By Proposition \ref{prop4}, Lemma \ref{lem52} and Minkowski's inequality, 
   \begin{align}\label{H21}
	H_{21}
	&\le C \|e^{-{\widetilde{c}_0}|\xi_h|^2t}\,\widehat{\p_2u}_0 \ \|_{L^2{(\R^3)}}
	+C\|e^{ -{ c_3}t } \ \widehat{\p_2 u}_0\|_{L^2{(\R^3)}}\nonumber\\
    &= C \Big\|\, \|e^{-{\widetilde{c}_0}\Lambda_2^2t}\,e^{-{\widetilde{c}_0}\Lambda_1^2t}  \p_2u_0 \|_{L^2_{x_2}} \ \Big\|_{L_{x_1x_3}^2}
	+C e^{ -{ c_3}t }\| \ \p_2 u_0\|_{L^2}\nonumber\\
    &\le C(1+t)^{-\f{3}{4}}\Big\| \, \|e^{-{\widetilde{c}_0}\Lambda_1^2t}  u_0 \|_{L^1_{x_2}} \ \Big\|_{L_{x_1x_3}^2}
	+C(1+t)^{-1}\|\p_2u_0\|_{L^2}\nonumber\\
    &\le C(1+t)^{-\f{3}{4}}\big\| \,\|e^{-{\widetilde{c}_0} \Lambda_1^2t}u_0\|_{L_{x_1}^2} \big\|_{ L_{x_3}^2L_{x_2}^1}
	+C(1+t)^{-1}\|\p_2u_0\|_{L^2}\nonumber\\
	&\le C(1+t)^{-1}(\|u_0\|_{L_{x_3}^2L_{x_1x_2}^1 }+\|\p_2u_0\|_{L^2}).
	\end{align}
Similarly,
 \begin{align}\label{H22}
	H_{22}
    \le C(1+t)^{-1}(\|b_0\|_{L_{x_3}^2L_{x_1x_2}^1}+\|\p_2b_0\|_{L^2}).
	\end{align}
For $H_{23}$, similarly to (\ref{Iu31}), we first bound it by
\begin{align}\label{IH13}
	H_{23}
	&\le C\int_0^{t-1} \|e^{-{\widetilde{ c}_0}|\xi_h|^2(t-\tau)}\widehat{\p_2M}_1(\tau)\|_{L^2{(\R^3)}}\,d\tau\notag\\
	&\quad + C\int_{0}^{t} (1+t-\tau)^{-m}\|\p_2M_1(\tau)\|_{L^2{(\R^3)}}\,d\tau,
\end{align}
where $M_1=b\cdot \na b-u\cdot \na u$. We  consider the first term involving $u\cdot\na u$ in (\ref{IH13}).
Firstly,  from the estimates (\ref{Eu}) and (\ref{Eut}), we obtain
\begin{align*}
\int_0^{t-1} &\|e^{-{\widetilde{ c}_0}|\xi_h|^2(t-\tau)}\widehat{\p_2(u\cdot\na u)}(\tau)\|_{L^2{(\R^3)}}\,d\tau
\le \int_0^{t}(1+t-\tau)^{-1} \|u\cdot\na u (\tau)\|_{L_{x_3}^2L_{x_1x_2}^1}\,d\tau\nonumber\\
&\le CE_2(t)  \int_0^{t} (1+t-\tau)^{-1} \Big[(1+\tau)^{-\f 32+\f 12\varepsilon}
     + (1+\tau)^{-\f {5} {4}+ \varepsilon}\Big]\,d\tau\nonumber\\
&\le  CE_2 (t) (1+t)^{-1}.
\end{align*}
For the second term in (\ref{IH13}), it follows from H\"{o}lder's inequality and Sobolev's inequality that
\begin{align}\label{1u}
\|\p_2(u\cdot\na u)\|_{L^2}
&\le \|\p_2u\|_{L^2}\|\na u\|_{L^{\infty}} + \|u_h\|_{L^{\infty}}\|\p_2\na_hu\|_{L^{2}} +\|u_3\|_{L^{4}}\|\p_2\p_3u\|_{L^{4}}\nonumber\\
&\le C (\|\p_2u\|_{L^2}\|\na^2u\|_{L^{2}}^{\f 12}\|\na^3u\|_{L^{2}}^{\f 12} + \|\na u_h\|_{L^2}^{\f 12}\|\na^2 u_h\|_{L^2}^{\f 12}\|\p_2\na_hu\|_{L^{2}}\nonumber\\
 &\quad+\|u_3\|_{L^{2}}^{\f 14} \|\na u_3\|_{L^{2}}^{\f 34} \|\p_2\p_3u\|_{L^{2}}^{\f 14}\|\na\p_2\p_3u\|_{L^{2}}^{\f 34}).
\end{align}
Therefore, for $m>2$, we derive
\begin{align}\label{T1u}
&\int_{0}^{t} (1+t-\tau)^{-m}\|\p_2(u\cdot\na u)(\tau)\|_{L^2{(\R^3)}}\,d\tau\nonumber\\
&\le C\sup_{0\le t\le \tau} (1+\tau)\|\p_2u(\tau)\|_{L^2} (1+\tau)^{\f 18}\|\na^2u(\tau)\|_{L^{2}}^{\f 12}\|\na^3u(\tau)\|_{L^{2}}^{\f 12}
     \int_{0}^{t} (1+t-\tau)^{-m}(1+\tau)^{-\f 9 8}\, d\tau\nonumber\\
& \,+ C\sup_{0\le t\le \tau} (1+\tau)^{\f 12(\f 12-\varepsilon)} \|\na u_h\|_{L^2}^{\f 12}  (1+\tau)^{\f 18 }\|\na^2 u_h\|_{L^2}^{\f 12} (1+\tau)^{\f 23-\varepsilon}\|\p_2\na_hu(\tau)\|_{L^{2}}\nonumber\\
&\, \quad   \times\int_{0}^{t} (1+t-\tau)^{-m}(1+\tau)^{-\f {25}{24}+\f 32\varepsilon}\, d\tau\nonumber\\
& \,+  C\sup_{0\le t \le \tau} (1+\tau)^{\f 18}\|u_3(\tau)\|_{L^{2}}^{\f 14}  (1+\tau)^{\f 34}\|\na u_3(\tau)\|_{L^{2}}^{\f 34}
           (1+\tau)^{\f 14(\f 23-\varepsilon)}\|\p_2\p_3u(\tau)\|_{L^{2}}^{\f 14}  \|\na^3u(\tau)\|_{L^{2}}^{\f 34}\nonumber\\
&   \quad\times\int_{0}^{t} (1+t-\tau)^{-m}(1+\tau)^{-\f {25}{24}+\f 14 \varepsilon}\, d\tau\nonumber\\
& \le C E_2^{\f 34}(t) E_0^{\f 14}(t)(1+t)^{-\f 98}+ C  E_2(t)(1+t)^{-\f {25}{24}+\f 32\varepsilon}
      + C E_2^{\f 58}(t) E_0^{\f 38}(t)(1+t)^{-\f {25}{24}+\f 14\varepsilon}\nonumber\\
&\le C E(t)(1+t)^{-1}.
\end{align}
Consequently,
\begin{align*}
	\int_0^t\|\widehat{Q}_1(t-\tau)\widehat{\p_2(u\cdot\na u)}(\tau)\|_{L^2{(\R^3)}}\,d\tau
	\le C E(t)(1+t)^{-1}.
\end{align*}
Similarly,
\begin{align*}
	\int_0^t\|\widehat{Q}_1(t-\tau)\widehat{\p_2(b\cdot\na b)}(\tau)\|_{L^2{(\R^3)}}\,d\tau
	\le C E(t)(1+t)^{-1}.
\end{align*}
Hence, 
\begin{align}\label{H23}
	H_{23}\le C E(t)(1+t)^{-1},
\end{align}
Similarly,
\begin{align}\label{H24}
	H_{24}\le C E(t)(1+t)^{-1}.
\end{align}
 (\ref{H21}), (\ref{H22}), (\ref{H23}) and (\ref{H24}) yield
\begin{align*}
	(1+t)\|\p_2u(t)\|_{L^2} \le C E(t) +C(\|(u_0,b_0)\|_{L_{x_3}^2L_{x_1x_2}^1}+\|(\p_2u_0,\p_2b_0)\|_{L^2}).
\end{align*}
Similarly,
\begin{align*}
	(1+t)\|\p_2b(t)\|_{L^2} \le C E(t) +C(\|(u_0,b_0)\|_{L_{x_3}^2L_{x_1x_2}^1}+\|(\p_2u_0,\p_2b_0)\|_{L^2}).
\end{align*}
For $i=1$, $\|(\p_1u, \p_1b)\|_{L^2}$ obeys a similar bound to $\|(\p_2u, \p_2b)\|_{L^2}$ with only a minor modification of (\ref{1u}) and (\ref{T1u}),
\begin{align*}
	(1+t)\|(\p_1u(t),\p_1b(t))\|_{L^2} \le C E(t) + C(\|(u_0,b_0)\|_{L_{x_3}^2L_{x_1x_2}^1}+\|(\p_1u_0,\p_1b_0)\|_{L^2}).
\end{align*}

 \vskip .1in
 ${\bf(2)}$ \,${\bf i=3}$
 \vskip .1in
 Invoking the estimate $(\ref{Bu0})$, we have
 \begin{align}\label{H31}
	H_{31}
	&\le C \|e^{-{\widetilde{c}_0}|\xi_h|^2t}\widehat{\p_3u}_0 \ \|_{L^2{(\R^3)}}
	+C\|e^{ -{ c_3}t } \ \widehat{ \p_3u}_0\|_{L^2{(\R^3)}}\nonumber\\
	&\le C(1+t)^{-\f{1}{2}}(\|\p_3u_0\|_{L_{x_3}^2L_{x_1x_2}^1}+\|\p_3u_0\|_{L^2}).
 \end{align}
 $H_{33}$ can be similarly estimated as $H_{23}$,
 \begin{align}\label{IH33}
	H_{33}
	&\le C\int_0^{t-1} \|e^{-{\widetilde{ c}_0}|\xi_h|^2(t-\tau)}\widehat{\p_3M}_1(\tau)\|_{L^2{(\R^3)}}\,d\tau\notag\\
	&\quad + C\int_{0}^{t} (1+t-\tau)^{-m}\|\p_3M_1(\tau)\|_{L^2{(\R^3)}}\,d\tau.
 \end{align}
Firstly, we have
 \begin{align*}
 \int_0^{t-1} \|e^{-{\widetilde{ c}_0}|\xi_h|^2(t-\tau)}\widehat{\p_3(u\cdot \na u)}(\tau)\|_{L^2{(\R^3)}}\,d\tau
 \le C \int_0^{t-1} (1+t-\tau)^{-\f 12 }\|\p_3(u\cdot \na u)(\tau)\|_{L_{x_3}^2L_{x_1x_2}^1} \,d\tau
 \end{align*}
 Applying the estimate (\ref{fg}) yields
 \begin{align}\label{B3u}
 \|\p_3(u\cdot \na u)\|_{L_{x_3}^2L_{x_1x_2}^1}
 &\le C(\|\p_3u_h\|_{L^2}^{\f 12} \|\p_3^2u_h\|_{L^2}^{\f 12}  \|\na_hu\|_{L^2}
 +\|\p_3u\|_{L^2}^{\f 12} \|\p_3^2u\|_{L^2}^{\f 12}  \|\p_3u_3\|_{L^2}\nonumber\\
 &+\|u_h\|_{L^2}^{\f 12} \|\p_3u_h \|_{L^2}^{\f 12}  \|\p_3\na_h u\|_{L^2}
 +\|u_3\|_{L^2}^{\f 12} \|\p_3u_3\|_{L^2}^{\f 12}  \|\p_3^2u\|_{L^2}).
\end{align}
As a consequence, we arrive at
\begin{align}\label{uH331}
&\int_0^{t-1} \|e^{-{\widetilde{ c}_0}|\xi_h|^2(t-\tau)}\widehat{\p_3(u\cdot \na u)}(\tau)\|_{L^2{(\R^3)}}\,d\tau\nonumber\\
&\le CE_2(t) \int_{0}^{t} (1+t-\tau)^{-\f 12}\big[(1+\tau)^{-\f {11} 8 + \f 12\varepsilon} +(1+\tau)^{-\f 7 6+\f 32\varepsilon } +(1+\tau)^{-1}\big]\, d\tau\nonumber\\
&\le C E(t)(1+t)^{-\f12+\varepsilon }.
\end{align}
To bound the second term in (\ref{IH33}), we apply H\"{o}lder's and Sobolev's inequalities to obtain
 \begin{align}\label{uH332}
&\int_{0}^{t} (1+t-\tau)^{-m}\|\p_3(u\cdot \na u)(\tau)\|_{L^2{(\R^3)}}\,d\tau\nonumber\\
&\le C\int_{0}^{t} (1+t-\tau)^{-m}(\|\p_3u(\tau)\|_{L^2} \|\na u(\tau)\|_{L^{\infty}}
+ \|u(\tau)\|_{L^{\infty}} \|\na\p_3 u(\tau)\|_{L^{2}} ) \,d\tau\nonumber\\
&\le C\int_{0}^{t} (1+t-\tau)^{-m}
(\|\p_3u(\tau)\|_{L^2} \|\na^2 u(\tau)\|_{L^2}^{\f 12}\|\na^3 u(\tau)\|_{L^2}^{\f 12}\nonumber \\
&\qquad \qquad\qquad \qquad
+ \|\na u(\tau)\|_{L^2}^{\f 12} \|\na^2 u(\tau)\|_{L^2}^{\f 12} \|\na\p_3 u(\tau)\|_{L^{2}} ) \,d\tau\nonumber\\
&\le C \Big(E_2^{\f 34}(t) E_0^{\f 14}(t)+E_2(t)\Big)
\int_{0}^{t} (1+t-\tau)^{-m}(1+\tau)^{-\f {5} 8+\varepsilon}\, d\tau\nonumber\\
&\le C E(t)(1+t)^{-\f 1 2}.
\end{align}
The estimates (\ref{uH331}) and (\ref{uH332}) then lead to
\begin{align*}
\int_0^{t} \| Q_1(t-\tau) \widehat{\p_3(u\cdot \na u)}(\tau)\|_{L^2{(\R^3)}}\,d\tau
\le  C E(t)(1+t)^{-\f 1 2+\varepsilon}.
\end{align*}
Therefore, 
\begin{align}\label{H33}
H_{33}\le  C E(t)(1+t)^{-\f 1 2+\varepsilon}.
\end{align}
Similarly,
\begin{align}\label{H3234}
H_{32}+H_{34}\le C\big(\|\p_3b_0\|_{L_{x_3}^2L_{x_1x_2}^1}+\|\p_3b_0\|_{L^2}+ E(t)\big)(1+t)^{-\f 1 2+\varepsilon}.
\end{align}
Finally, by the estimates (\ref{H31}), (\ref{H33}) and (\ref{H3234}), we conclude
\begin{align*}
	(1+t)^{\f 12-\varepsilon}\|\p_3u(t)\|_{L^2}
	\le  C E(t)+C(\|(\p_3u_0,\p_3b_0)\|_{L_{x_3}^2L_{x_1x_2}^1}+\|(\p_3u_0,\p_3b_0)\|_{L^2}).
\end{align*}
This completes the proof of Lemma \ref{lem56}.
\end{proof}

\vskip .1in
Next we bound $E_{23}(t)$, which involves the second-order derivatives of $(u,b)$.
\begin{lemma} \label{lem57}
Let $(u,b)$ be a solution  to (\ref{MHD1}). Then it holds
 \begin{align}\label{E23}
  E_{23}(t)&\le C E^2(t)+C\big(\|(u_0,b_0)\|_{L_{x_3}^2L_{x_1x_2}^1}^2
   +(\|(\p_3u_0,\p_3b_0)\|_{L_{x_3}^2L_{x_1x_2}^1}^2\nonumber\\
  &\quad+\|(\p_3^2u_0,\p_3^2b_0)\|_{L_{x_3}^2L_{x_1x_2}^1}^2
  +\|(\Delta u_0,\Delta b_0)\|_{L^2}^2).
 \end{align}
 \end{lemma}

\vskip .1in
\begin{proof}[Proof of Theorem \ref{lem57}]
	First of all, we have, for $i,j=1,2,3$,
	\ben
    \widehat{\p_i\p_j u}(\xi, t) &=& \widehat{Q}_1(t)\widehat{\p_i\p_j u}_0+\widehat{Q}_2(t)\widehat{\p_i\p_j b}_0\notag\\
	&& +\int_0^t\big(\widehat{Q}_1(t-\tau)\widehat{\p_i\p_j N}_1(\tau)
	+\widehat{Q}_2(t-\tau)\widehat{\p_i\p_j N}_2(\tau)\big)\,d\tau, \label{uin2}\\
	\widehat{\p_i\p_j b}(\xi, t) &=& \widehat{Q}_2(t)\widehat{\p_i\p_j u}_0+\widehat{Q}_3(t)\widehat{\p_i\p_j b}_0 \notag \\
	&& +\int_0^t\big(\widehat{Q}_2(t-\tau)\widehat{\p_i\p_j N}_1(\tau)
	+\widehat{Q}_3(t-\tau)\widehat{\p_i\p_j N}_2(\tau)\big)\,d\tau. \notag
	\een
  Throughout the proof, we only show the bound of $\|\p_i\p_j u(t)\|_{L^2}$. The estimates for $\|\p_i\p_j b(t)\|_{L^2}$ can be obtained similarly.
  Taking the $L^2$ norm on both side of (\ref{uin2}), we have
	\begin{align*}
	\|\p_i\p_ju(t)\|_{L^2(\R^3)}&=\|\widehat{\p_i\p_ju}( t)\|_{L^2(\R^3)}
	\le\|\widehat{Q}_1(t)\widehat{\p_i\p_ju}_0\|_{L^2(\R^3)}+\|\widehat{Q}_2(t)\widehat{\p_i\p_jb}_0\|_{L^2(\R^3)}\nonumber\\
	&+\int_0^t\|\widehat{Q}_1(t-\tau)\widehat{\p_i\p_jN}_1(\tau)\|_{L^2(\R^3)}\,d\tau
	+\int_0^t\|\widehat{Q}_2(t-\tau)\widehat{\p_i\p_jN}_2(\tau)\|_{L^2(\R^3)}\,d\tau.\nonumber\\
	&=K_{ij1}+K_{ij2}+K_{ij3}+K_{ij4}.
    \end{align*}
   We focus on  $K_{ij1}$ and $K_{ij3}$. The bound for the other terms can be established in a similar way. The proof will be split into four cases: $i=1,j=1,2$;\;$i=1,j=3$;\; $i=2,j=2,3$;\; $i=j=3$.

   \vskip .1in
  ${\bf(1)}$ \,${\bf i=1,\,j=1,2}$.
  \vskip .1in
   It suffices to investigate the case $i=1,j=2$. The case  $i=1,j=1$ can be dealt with similarly. By Lemma \ref{lem52},
    \begin{align}\label{K121}
	K_{121}
	&\le C \|e^{-{\widetilde{c}_0}|\xi_h|^2t}|\xi_h|^2 \widehat{u}_0  \|_{L^2{(\R^3)}}
	+C\|e^{ -{ c_3}t } \ \widehat{ \p_1\p_2u}_0\|_{L^2{(\R^3)}}\nonumber\\
	&\le C(1+t)^{-\f{3}{2}}(\|u_0\|_{L_{x_3}^2L_{x_1x_2}^1}+\|\p_1\p_2u_0\|_{L^2}).
   \end{align}
Similarly,
   \begin{align}\label{K122}
	K_{122}
	\le C(1+t)^{-\f{3}{2}}(\|b_0\|_{L_{x_3}^2L_{x_1x_2}^1}+\|\p_1\p_2b_0\|_{L^2}).
   \end{align}
   For $K_{123}$, we first give a different bound from the ones in Lemma \ref{lem55} and Lemma \ref{lem56}.
   \begin{align*}
	K_{123}&\le C\int_0^{t-1} \|e^{-\widetilde{c}_0 |\xi_h|^2(t-\tau)}\widehat{\p_1\p_2 M}_1(\tau)\|_{L^2{(\R^3)}}\,d\tau
	+ C\int_{t-1}^{t} \|e^{-\widetilde{c}_0|\xi_h|^2(t-\tau)}\widehat{\p_1\p_2 M}_1(\tau)\|_{L^2{(\R^3)}}\,d\tau\nonumber\\
   & +  C\int_{0}^{t} e^{-c_3(t-\tau)}\|e^{-c_2 \xi_1^2(t-\tau)}\widehat{\p_1\p_2 M}_1(\tau)\|_{L^2{(\R^3)}}\,d\tau.
   \end{align*}
   For $\tau\in[t-1,t]$, we have $e^{-c_3(t-\tau)}\ge e^{-c_3}$ and thus 
   $$
   \int_{t-1}^{t} \|e^{-\widetilde{c}_0|\xi_h|^2(t-\tau)}\widehat{\p_1\p_2 M}_1(\tau)\|_{L^2{(\R^3)}}\,d\tau \le e^{c_3} \int_{t-1}^{t} e^{-c_3(t-\tau)} \|e^{-\widetilde{c}_0|\xi_h|^2(t-\tau)}\widehat{\p_1\p_2 M}_1(\tau)\|_{L^2{(\R^3)}}\,d\tau.
   $$
   As a consequence, for a constant $c_4>0$,
   \begin{align*}
	K_{123}&\le C\int_0^{t-1} \|e^{-\widetilde{c}_0 |\xi_h|^2(t-\tau)}\widehat{\p_1\p_2 M}_1(\tau)\|_{L^2{(\R^3)}}\,d\tau\\
    &\quad +  C\int_{0}^{t} e^{-c_3(t-\tau)}\|e^{-c_4 \xi_1^2(t-\tau)}\widehat{\p_1\p_2 M}_1(\tau)\|_{L^2{(\R^3)}}\,d\tau\\
   &:=K_{1231}+K_{1232}.
   \end{align*}
  Invoking (\ref{K121}), (\ref{Eu}) and (\ref{Eut}), we have
  \begin{align*}
  &\int_0^{t-1} \|e^{-{\widetilde{ c}_0}|\xi_h|^2(t-\tau)} \widehat {\p_1\p_2(u\cdot \na u)}(\tau) \|_{L^2{(\R^3)}}\,d\tau\\
  & \le C \int_0^{t} (1+t-\tau)^{-\f 32}\|u\cdot \na u (\tau) \|_{ L_{x_3}^2L_{x_1x_2}^1}\,d\tau\\
  &\le C E_2(t)  \int_0^{t} (1+t-\tau)^{-\f 32} \Big((1+\tau)^{-\f 3 2+\f 12 \varepsilon}+(1+\tau)^{-\f 5 4+ \varepsilon} \Big)            \,d\tau\\
  &\le C E (t)(1+t)^{-\f 5 4+\varepsilon}.
  \end{align*}
  Hence, 
 \begin{align}\label{K1231}
  K_{1231} \le C E(t) (1+t)^{-\f 5 4+ \varepsilon}.
 \end{align}
For $K_{1232}$, according to Lemma \ref{lem52}, we have
   \begin{align*}
	K_{1232}\le  C\int_{0}^{t} e^{-c_3(t-\tau)}(t-\tau)^{-\f 12 }\|\p_2 (u\cdot\na u-b\cdot\na b)(\tau)\|_{L^2{(\R^3)}}\,d\tau.
   \end{align*}
By the anisotropic inequality (\ref{Ine3}),
   \begin{align*}
   \|\p_2(u\cdot\na u)\|_{L^2{(\R^3)}}
   &\le C \|\p_2u\|_{L^2}^{\f 14} \|\p_2^2u\|_{L^2}^{\f 14} \|\p_2\p_3u\|_{L^2}^{\f 14} \|\p_2^2\p_3u\|_{L^2}^{\f 14}
   \|\na u\|_{L^2}^{\f 12}\|\p_1\na u\|_{L^2}^{\f 12}\\
   &+ C \|u\|_{L^2}^{\f 14} \|\p_1u\|_{L^2}^{\f 14} \|\p_2u\|_{L^2}^{\f 14} \|\p_1\p_2u\|_{L^2}^{\f 14}
    \|\na\p_2 u\|_{L^2}^{\f 12}\|\p_2\p_3\na u\|_{L^2}^{\f 12}.
   \end{align*}
  Hence,
  \begin{align*}
  &\int_{0}^{t} e^{-c_3(t-\tau)}(t-\tau)^{-\f 12 }\|\p_2 (u\cdot\na u)(\tau)\|_{L^2{(\R^3)}}\,d\tau\\
  &\le C E_2^{\f 78}(t) E_0^{\f 18}(t)  \int_0^{t}  e^{-c_3(t-\tau)}(t-\tau)^{-\f 12}\, (1+\tau)^{-\f 43 +\f 32\varepsilon }\,d\tau\nonumber\\
  &\quad+ C  E_2^{\f 34}(t) E_0^{\f 14}(t)  \int_0^{t}  e^{-c_3(t-\tau)}(t-\tau)^{-\f 12}\, (1+\tau)^{-\f {61 }{48} +\f 34\varepsilon }\,d\tau\\
  &\le C E(t)  \int_0^{t}  e^{-\f {c_3}{2}(t-\tau)}(t-\tau)^{-\f 12}\,
    (1+t-\tau)^{-m}(1+\tau)^{-\f {5 }{4} +\varepsilon }\,d\tau,
  \end{align*}
  where we have used the simple fact: $e^{-ct}(1+t)^m\le C(m)$ for any $t\ge 0, m\ge 0$.
Furthermore, selecting $m>2$, and then applying H\"{o}lder inequality with $1<p<2$ and $\f 1p+ \f 1q=1$, we infer
  \begin{align}\label{Gamma}
    &\int_{0}^{t} e^{-c_3(t-\tau)}(t-\tau)^{-\f 12 }\|\p_2(u\cdot \na u) (\tau)\|_{L^2{(\R^3)}}\,d\tau\nonumber\\
    &\le C E(t)
    \Big( \int_0^t e^{-\f {c_3p}{2}(t-\tau)}(t-\tau)^{-\f p 2}d\tau \Big)^{\f 1 p}\,
    \Big(\int_0^t (1+t-\tau)^{-mq}(1+\tau)^{(-\f 54+\varepsilon)q}d\tau\Big)^{\f 1 q}\nonumber\\
    & \le CE(t)(1+t)^{-\f 5 4 +\varepsilon}.
  \end{align}
  where we have used fact that the integration $\int_0^{\infty} x^{s-1} e^{-x}dx\, (s>0)$ converges to  $\Gamma(s)$.
 Consequently, 
   \begin{align}\label{K1232}
	K_{1232}\le  C E(t)(1+t)^{-\f 5 4 +\varepsilon}.
   \end{align}
(\ref{K1231}) and (\ref{K1232}) lead to
   \begin{align}\label{K124}
	K_{123}\le  C E(t)(1+t)^{-\f 5 4 +\varepsilon}.
   \end{align}
With a similar argument, we obtain
   \begin{align}\label{K123}
	K_{124}\le  C E(t)(1+t)^{-\f 5 4 +\varepsilon}.
   \end{align}
  Combining the estimates (\ref{K121}), (\ref{K122}), (\ref{K123}) and (\ref{K124}), we derive
   \begin{align*}
   (1+t)^{\f 5 4 -\varepsilon}\|\p_1\p_2u(t)\|_{L^2}
	\le C E(t)+C(\|(u_0,b_0)\|_{L_{x_3}^2L_{x_1x_2}^1}+\|(\p_1\p_2u_0,\p_1\p_2b_0)\|_{L^2}).
   \end{align*}
   Similarly, we can also obtain
    \begin{align*}
   (1+t)^{\f 5 4 -\varepsilon}\|\p_1^2u(t)\|_{L^2}
	\le C E(t)+C(\|(u_0,b_0)\|_{L_{x_3}^2L_{x_1x_2}^1}+\|(\p_1^2u_0,\p_1^2b_0)\|_{L^2}).
   \end{align*}

   \vskip .1in
  ${\bf(2)}$ \,${\bf i=1,j=3}$.
  \vskip .1in
   Firstly, from (\ref{H21}), we have
   \begin{align}\label{K131}
	K_{131}\le C(1+t)^{-1}(\|\p_3u_0\|_{L_{x_3}^2L_{x_1x_2}^1}+\|\p_1\p_3u_0\|_{L^2}).
   \end{align}
   For $K_{133}$, similarly to $K_{123}$, we first bound it as
  \begin{align*}
	K_{133}&\le C\int_0^{t-1} \|e^{-\widetilde{c}_0 |\xi_h|^2(t-\tau)}\widehat{\p_1\p_3 M}_1(\tau)\|_{L^2{(\R^3)}}\,d\tau\\
    &\quad +  C\int_{0}^{t} e^{-c_3(t-\tau)}\|e^{-c_4 \xi_1^2(t-\tau)}\widehat{\p_1\p_3 M}_1(\tau)\|_{L^2{(\R^3)}}\,d\tau\\
   &\le C\int_0^{t} (1+t-\tau)^{-1}\|\p_3 M_1(\tau)\|_{L_{x_3}^2L_{x_1x_2}^1}\,d\tau\\
    &\quad +  C\int_{0}^{t} e^{-c_3(t-\tau)}(t-\tau)^{-\f 12 }\|\p_3 M_1(\tau)\|_{L^2{(\R^3)}}\,d\tau\\
   &:=K_{1331}+K_{1332}.
   \end{align*}
   Invoking (\ref{B3u}) and (\ref{uH331}), we get
   \begin{align}\label{K1331}
   K_{1331}
   &\le CE_2(t) \int_{0}^{t} (1+t-\tau)^{-1}\Big[(1+\tau)^{-\f {11} 8 + \f 12\varepsilon} +(1+\tau)^{-\f 7 6+\f 32\varepsilon } +(1+\tau)^{-1}\,\Big] d\tau\nonumber\\
   &\le C E(t)(1+t)^{-1+\varepsilon }.
   \end{align}
   For $K_{1332}$,  by H\"{o}lder's inequality and Sobolev's inequality, we first have
   \begin{align*}
   \|\p_3(u\cdot\na u)\|_{L^2}
   &\le \|\p_3u_j\|_{L^4}\|\p_j u\|_{L^4}+\|u_j\|_{L^{\infty}}\|\p_j\p_3 u\|_{L^2}\\
   &\le C\|\p_3u_h\|_{L^2}^{\f 14}\|\p_3\na u_h\|_{L^2}^{\f 34}\|\na_h u\|_{L^2}^{\f 14}\|\na\na_h u\|_{L^2}^{\f 34}\\
        &\quad +C\|\p_3u_3\|_{L^2}^{\f 14}\|\p_3\na u_3\|_{L^2}^{\f 34}\|\p_3 u\|_{L^2}^{\f 14}\|\na\p_3 u\|_{L^2}^{\f 34}\\
        &\quad +C\|\na u_h\|_{L^2}^{\f 12}\|\na^2 u_{h}\|_{L^2}^{\f 12}\|\na_h\p_3 u\|_{L^2}
        +C\|\na u_3\|_{L^2}^{\f 12}\|\na^2 u_{3}\|_{L^2}^{\f 12}\|\p_3^2 u\|_{L^2}.
  \end{align*}
  Then, for $m>1$,
  \begin{align*}
    &\int_{0}^{t} e^{-c_3(t-\tau)}(t-\tau)^{-\f 12 }\|\p_3(u\cdot \na u) (\tau)\|_{L^2}\,d\tau\\
    &\le C  E_2 (t)\int_0^t e^{-c_3(t-\tau)}(t-\tau)^{-\f 12 }
    \Big[(1+\tau)^{-\f{17}{16}+\varepsilon}+ (1+\tau)^{-\f{25}{24}+\f 32\varepsilon} + (1+\tau)^{-\f{13}{12}+\f 12\varepsilon}\Big]  d\tau\\
    &\le C E(t) \int_0^t e^{-c_3(t-\tau)}(t-\tau)^{-\f 12 }(1+\tau)^{-1+\varepsilon}d\tau\\
    &\le C E (t)\int_0^t e^{-\f {c_3}{2}(t-\tau)}(t-\tau)^{-\f 12 }\, (1+t-\tau)^{-m}(1+\tau)^{-1+\varepsilon}d\tau\\
    &\le C E(t)(1+t)^{-1+\varepsilon},
  \end{align*}
 where we have used a similar derivation with $(\ref{Gamma})$ for the last inequality.
 Thus, we get
   \begin{align*}
    K_{1332}\le C E (t)(1+t)^{-1+\varepsilon}.
   \end{align*}
  which, together with (\ref{K1331}), gives
  \begin{align}\label{K133}
    K_{133}\le C E (t)(1+t)^{-1+\varepsilon}.
   \end{align}
 Therefore, by (\ref{K131}) and (\ref{K133}),  we conclude
  \begin{align*}
  (1+t)^{1-\varepsilon}\|\p_1\p_3u(t)\|_{L^2}
	\le CE(t)+C(\|(\p_3u_0,\p_3b_0)\|_{L_{x_3}^2L_{x_1x_2}^1}+\|(\p_1\p_3u_0,\p_1\p_3b_0)\|_{L^2}).
  \end{align*}

  \vskip .1in
  ${\bf(3)}$ \,${\bf i=2,j=2,3}$.
  \vskip .1in
   It suffices to bound $\|\p_2\p_3 u\|_{L^2}$. Firstly, a similar argument with (\ref {H21}) yields
   \begin{align}\label{K231}
	K_{231}
	&\le C \|e^{-{\widetilde{c}_0}|\xi_h|^2t}\widehat{\p_2\p_3u}_0 \ \|_{L^2{(\R^3)}}
	+C\|e^{ -{ c_3}t } \ \widehat{ \p_2\p_3u}_0\|_{L^2{(\R^3)}}\nonumber\\
	&\le C(1+t)^{-1}(\|\p_3u_0\|_{L_{x_3}^2L_{x_1x_2}^1}+C\|\p_2\p_3u_0\|_{L^2}).
   \end{align}
   As in $H_{23}$, $K_{233}$ is firstly bounded by
   \begin{align*}
	K_{233}
	&\le C\int_0^{t-1} \|e^{-{\widetilde{ c}_0}|\xi_h|^2(t-\tau)}\widehat{\p_2\p_3 M}_1(\tau)\|_{L^2{(\R^3)}}\,d\tau\\
	&\quad + C\int_{0}^{t} (1+t-\tau)^{-m}\|\p_2\p_3 M_1(\tau)\|_{L^2{(\R^3)}}\,d\tau\nonumber\\
    &\le C\int_0^{t} (1+t-\tau)^{-1 }\|\p_3 M_1(\tau)\|_{L_{x_3}^2L_{x_1x_2}^1}\,d\tau\\
	&\quad + C\int_{0}^{t} (1+t-\tau)^{-m}\|\p_2\p_3 M_1(\tau)\|_{L^2{(\R^3)}}\,d\tau\nonumber\\
    &:=K_{2331}+K_{2332}.
   \end{align*}
   Now we estimate $K_{2331}$. Recalling the bound (\ref{K1331}) gives
   \begin{align}\label{K2331}
   K_{2331}\le C E_2(t)(1+t)^{-1+\varepsilon }.
   \end{align}
   By  H\"{o}lder's inequality and Sobolev's inequality, 
   \begin{align*}
   \|\p_2\p_3(u\cdot\na u)\|_{L^2}
   &\le \|\p_2\p_3u\cdot\na u\|_{L^2}+\|\p_2u\cdot\na\p_3u\|_{L^2}+\|\p_3u\cdot\na\p_2 u\|_{L^2}+\|u\cdot\na\p_2\p_3u\|_{L^2}\\
   &\le  \|\na u\|_{\infty}\|\na\p_2 u\|_{L^2}+\|\p_2 u\|_{L^4}\|\na\p_3 u\|_{L^4}+\|u\cdot\na\p_2\p_3u\|_{L^2}\\
   &\le C \|\na^2 u\|_{L^2}^{\f 12}\|\na^3 u\|_{L^2}^{\f 12}\|\na\p_2 u\|_{L^2}
   +\|\p_2 u\|_{L^2}^{\f 14}\|\p_2\na u\|_{L^2}^{\f 34}\|\na\p_3 u\|_{L^2}^{\f 14}\|\na^2\p_3 u\|_{L^2}^{\f 34}\\
   &+\|u\|_{L^2}^{\f 14}\|\p_2u\|_{L^2}^{\f 14}\|\p_3u\|_{L^2}^{\f 14}\|\p_2\p_3u\|_{L^2}^{\f 14}\|\na\p_2\p_3 u\|_{L^2}^{\f 12}\|\na\p_1\p_2\p_3 u\|_{L^2}^{\f 12},
  \end{align*}
  where we have used the anisotropic inequality (\ref{Ine3}) for $\|u\cdot\na\p_2\p_3u\|_{L^2}$.
  Thus,
   \begin{align*}
   &\int_{0}^{t} (1+t-\tau)^{-m}\|\p_2\p_3 (u\cdot \na u)(\tau)\|_{L^2{(\R^3)}}\,d\tau \\
   &\le C\int_0^t (1+t-\tau)^{-m} \Big(\|\na^2 u\|_{L^2}^{\f 12}\|\na^3 u\|_{L^2}^{\f 12}\|\na\p_2 u\|_{L^2}
   +\|\p_2 u\|_{L^2}^{\f 14}\|\p_2\na u\|_{L^2}^{\f 34}\|\na\p_3 u\|_{L^2}^{\f 14}\|\na^2\p_3 u\|_{L^2}^{\f 34}\Big)\, d\tau\\
   &\quad+C\int_0^t (1+t-\tau)^{-m}\|u\|_{L^2}^{\f 14}\|\p_2u\|_{L^2}^{\f 14}\|\p_3u\|_{L^2}^{\f 14}\|\p_2\p_3u\|_{L^2}^{\f 14}\|\na\p_2\p_3 u\|_{L^2}^{\f 12}\|\na\p_1\p_2\p_3 u\|_{L^2}^{\f 12}\, d \tau\\
   &:=L_1+L_2.
  \end{align*}
  By means of (\ref{Ine5}), for $m>1$, we infer
  \begin{align*}
  L_1
  &\le C \sup_{0\le\tau\le t} (1+\tau)^{\f 23-\varepsilon}\|\na\p_2 u(\tau)\|_{L^2}\|\na^2u(\tau)\|_{H^1}
         \int_0^t (1+t-\tau)^{-m} (1+\tau)^{-\f23+\varepsilon} d\tau\\
  &+ \sup_{0\le\tau\le t} (1+\tau)^{\f 14}\|\p_2u(\tau)\|_{L^2}^{\f 14} (1+\tau)^{\f 34(\f 23-\varepsilon)} \|\p_2\na u(\tau)\|_{L^2}^{\f 34}
   \|\na^2u(\tau)\|_{H^1} \\
   &\qquad \times \int_0^t (1+t-\tau)^{-m} (1+\tau)^{-\f 34+\f 34\varepsilon} d\tau\\
   &\le E_2^{\f 12}(t)E_{0}^{\f 12}(t)(1+t)^{-\f 23+\varepsilon}
    \le C E(t)(1+t)^{-\f 23+\varepsilon}.
  \end{align*}
  For $L_2$, applying H\"{o}lder's inequality yields, for $m>1$,
  \begin{align*}
  L_2
  &\le C \sup_{0\le\tau\le t} (1+\tau)^{\f 18}\|u(\tau)\|_{L^2}^{\f 14}  (1+\tau)^{\f 14}\|\p_2u(\tau)\|_{L^2}^{\f 14}
                    (1+\tau)^{\f 18-\f 14\varepsilon}\|\p_3u(\tau)\|_{L^2}^{\f 14}\\
                   &\qquad \times  (1+\tau)^{\f 16-\f 14\varepsilon}\|\p_2\p_3u(\tau)\|_{L^2}^{\f 14}\\
        &\qquad\times \|\na\p_2\p_3 u(\tau)\|_{L^2}^{\f 12} \int_0^t (1+t-\tau)^{-m} (1+\tau)^{-\f23+\f 12\varepsilon}
        \|\na\p_1\p_2\p_3 u(\tau)\|_{L^2}^{\f 12} d\tau\\
  &\le CE_2^{\f 12}(t) E_0^{\f 14}(t) \Big(\int_0^t (1+t-\tau)^{-\f 43m} (1+\tau)^{-\f 8 9+\f 23\varepsilon} d\tau\Big)^{\f 34}
        \Big( \int_0^t \|\na\p_1\p_2\p_3 u(\tau)\|_{L^2}^{2} d\tau\Big)^{\f 14}\\
  & \le  CE_2^{\f 12}(t)E_0^{\f 12}(t) (1+t)^{-\f23+\f 12\varepsilon}\le C E(t)(1+t)^{-\f23+\varepsilon}.
  \end{align*}
Therefore,
  \begin{align*}
  \int_{0}^{t} (1+t-\tau)^{-m}\|\p_2\p_3 (u\cdot \na u)(\tau)\|_{L^2{(\R^3)}}\,d\tau
  \le  C E(t) (1+t)^{-\f23+\varepsilon}.
  \end{align*}
Thus, 
  \begin{align*}
   K_{2332}&\le  C E(t) (1+t)^{-\f23+\varepsilon},
  \end{align*}
which, together with (\ref{K2331}), gives
  \begin{align}\label{K233}
   K_{233}&\le  C E(t) (1+t)^{-\f23+\varepsilon}.
  \end{align}
 $K_{232}$ and  $K_{234}$ can be bounded with  similar arguments as  those for $K_{231}$ and $K_{233}$, respectively. Therefore, by (\ref{K231}) and (\ref{K233}), we conclude
  \begin{align*}
   (1+t)^{\f23-\varepsilon}\|\p_2\p_3 u(t)\|_{L^2}
   &\le  C E(t)+ C(\|(\p_3u_0,\p_3b_0)\|_{L_{x_3}^2L_{x_1x_2}^1}
   +\|(\p_2\p_3u_0,\p_2\p_3b_0)\|_{L^2}).
  \end{align*}
Similarly, 
  \begin{align*}
   (1+t)^{\f23-\varepsilon}\|\p_2^2 u(t)\|_{L^2}
   &\le  C E(t)+ C(\|(u_0,b_0)\|_{L_{x_3}^2L_{x_1x_2}^1}
   +\|(\p_2^2u_0,\p_2^2b_0)\|_{L^2}).
  \end{align*}

   \vskip .1in
  ${\bf(4)}$ \,${\bf i=j=3}$.
  \vskip .1in
  Firstly, we have
   \begin{align}\label{K331}
	K_{331}
	&\le C \|e^{-{\widetilde{c}_0}|\xi_h|^2t}\widehat{\p_3^2u}_0 \ \|_{L^2{(\R^3)}}
	+C\|e^{ -{ c_3}t } \ \widehat{ \p_3^2u}_0\|_{L^2{(\R^3)}}\nonumber\\
	&\le C(1+t)^{-\f{1}{2}}(\|\p_3^2u_0\|_{L_{x_3}^2L_{x_1x_2}^1}+\|\p_3^2u_0\|_{L^2}).
   \end{align}
   $K_{233}$ can be bounded as
   \begin{align*}
	K_{333}
	&\le C\int_0^{t-1} \|e^{-{\widetilde{ c}_0}|\xi_h|^2(t-\tau)}\widehat{\p_3^2 M}_1(\tau)\|_{L^2{(\R^3)}}\,d\tau
	+ C\int_{0}^{t} (1+t-\tau)^{-m}\|\p_3^2 M_1(\tau)\|_{L^2{(\R^3)}}\,d\tau\nonumber\\
    &\le C\int_0^{t} (1+t-\tau)^{-\f 12 }\|\p_3^2 M_1(\tau)\|_{L_{x_3}^2L_{x_1x_2}^1}\,d\tau
	+ C\int_{0}^{t} (1+t-\tau)^{-m}\|\p_3^2 M_1(\tau)\|_{L^2{(\R^3)}}\,d\tau\nonumber\\
    &:=K_{3331}+K_{3332}.
   \end{align*}
   We consider the integral
   \begin{align}\label{K3331u}
   \int_0^{t} (1+t-\tau)^{-\f 12 }\|\p_3^2(u\cdot\na u)(\tau)\|_{L_{x_3}^2L_{x_1x_2}^1}\,d\tau.
   \end{align}
   It follows from (\ref{fg}) that
   \begin{align}\label{Bu}
   &\|\p_3^2(u\cdot \na u)\|_{L_{x_3}^2 L_{x_1x_2}^1 }
   =\|\p_3^2u_j\,\p_j u + 2\p_3u_j\, \p_j\p_3u+u_j\,\p_j\p_3^2 u\|_{L_{x_3}^2 L_{x_1x_2}^1 }\nonumber\\
   &\le C(\|\p_3u\|_{L^2}^{\f 12} \|\p_3^2u\|_{L^2}^{\f 12}  \|\p_3^2u_3 \|_{L^2}
    +\|\na_h u\|_{L^2}^{\f 12} \|\p_3\na_h u\|_{L^2}^{\f 12}  \|\p_3^2u_h\|_{L^2}\nonumber\\
  &\quad+\|\p_3u_3\|_{L^2}^{\f 12} \|\p_3^2u_3\|_{L^2}^{\f 12}  \|\p_3^2u\|_{L^2}
    +\|\p_3 u_h\|_{L^2}^{\f 12} \|\p_3^2 u_h\|_{L^2}^{\f 12}  \|\p_3\na_hu\|_{L^2}\nonumber\\
   &\quad +\|u_3\|_{L^2}^{\f 12} \|\p_3u_3\|_{L^2}^{\f 12}  \|\p_3^3u\|_{L^2}
      +\|u_h\|_{L^2}^{\f 12} \|\p_3u_h\|_{L^2}^{\f 12}  \|\p_3^2\na_h u\|_{L^2})\nonumber\\
  &\le C(\|\na_h u\|_{L^2}^{\f 12} \|\p_3\na_h u\|_{L^2}^{\f 12}  \|\p_3^2u\|_{L^2}
    +\|\p_3 u\|_{L^2}^{\f 12} \|\p_3^2 u\|_{L^2}^{\f 12}  \|\p_3\na_hu\|_{L^2}\nonumber\\
   &\quad +\|u_3\|_{L^2}^{\f 12} \|\p_3u_3\|_{L^2}^{\f 12}  \|\p_3^3u\|_{L^2}
      +\|u_h\|_{L^2}^{\f 12} \|\p_3u_h\|_{L^2}^{\f 12}  \|\p_3^2\na_h u\|_{L^2}).
  \end{align}
  Inserting (\ref{Bu}) in (\ref{K3331u}), and using Lemma \ref{lem53}, the first three terms can be bounded by
  \begin{align*}
  &\int_0^t (1+t-\tau)^{-\f 12 }\big(\|\na_h u(\tau)\|_{L^2}^{\f 12} \|\p_3\na_h u(\tau)\|_{L^2}^{\f 12}  \|\p_3^2u(\tau)\|_{L^2}\\
  &+\|\p_3u(\tau)\|_{L^2}^{\f 12} \|\p_3^2u(\tau)\|_{L^2}^{\f 12}  \|\p_3\na_hu (\tau)\|_{L^2}\\
  &\quad+\|u_3(\tau)\|_{L^2}^{\f 12} \|\p_3u_3(\tau)\|_{L^2}^{\f 12}  \|\p_3^3u(\tau)\|_{L^2}\big) d\tau\\
  &\le C E_2(t)\int_0^t (1+t-\tau)^{-\f 12 }\Big((1+\tau)^{-\f{13}{12}+\f {\varepsilon}{2}}d\tau
  +(1+\tau)^{-\f{25}{24}+\f {3\varepsilon}{2}}\Big)d\tau\\
  &\quad+CE_2^{\f 12}(t)E_0^{\f 12}(t)\int_0^t (1+t-\tau)^{-\f 12 }(1+\tau)^{-\f{3}{4}}d\tau \\
  &\le C E (t)\Big((1+t)^{-\f 12 }+(1+t)^{-\f 14 }\Big).
  \end{align*}
 The last term needs more subtle estimates. We resort to H\"{o}lder's inequality and the integrability of $\|\p_3^2\na_h u\|_{L^2}$.
   \begin{align*}
  &\int_0^t (1+t-\tau)^{-\f 12 }\|u_h(\tau)\|_{L^2}^{\f 12} \|\p_3u_h(\tau)\|_{L^2}^{\f 12}  \|\p_3^2\na_h u(\tau)\|_{L^2}d\tau\\
  &\le C E_2^{\f 12}(t)\int_0^t (1+t-\tau)^{-\f 12 }(1+\tau)^{-\f{1}{2}+\f {\varepsilon}{2}}\|\p_3^2\na_h u(\tau)\|_{L^2}d\tau \\
  &\le C E_2^{\f 12}(t)  \Big(\int_0^t (1+t-\tau)^{-1 }(1+\tau)^{-1+\varepsilon}d\tau\Big)^{\f 12}
  \Big(\int_0^t\|\p_3^2\na_h u(\tau)\|_{L^2}^2d\tau\Big)^{\f 12} \\
  &\le C E_2^{\f 12}(t) E_0^{\f 1 2}(t)(1+t)^{-\f 14 }.
   \end{align*}
  Combining all the estimates above, we get
  \begin{align*}
  \int_0^{t} (1+t-\tau)^{-\f 12 }\|\p_3^2(u\cdot\na u)(\tau)\|_{L_{x_3}^2L_{x_1x_2}^1}\,d\tau
  \le C E(t)(1+t)^{-\f 14}.
  \end{align*}
  Thus,
  \begin{align}\label{K3331}
	K_{3331}
	\le C E(t)(1+t)^{-\f 14}.
   \end{align}
  Finally, applying  H\"{o}lder's inequality and Sobolev's inequality, for $m>1$, we infer
  \begin{align*}
  &\int_{0}^{t} (1+t-\tau)^{-m}\|\p_3^2(u\cdot \na u)(\tau)\|_{L^2{(\R^3)}}\,d\tau\\
  &\le C\int_{0}^{t} (1+t-\tau)^{-m}(\|\na^2 u(\tau)\|_{L^4} \|\na u(\tau)\|_{L^{4}}
  + \|u(\tau)\|_{L^{\infty}} \|\na\p_3^2 u(\tau)\|_{L^{2}} ) \,d\tau\nonumber\\
  &\le C\int_{0}^{t} (1+t-\tau)^{-m}
  ( \|\na^2 u(\tau)\|_{L^2}\|\na^3 u(\tau)\|_{L^2}^{\f 34}\|\na u(\tau)\|_{L^2}^{\f 14}\\
  &\qquad \qquad\qquad \qquad + \|\na u(\tau)\|_{L^2}^{\f 12} \|\na^2 u(\tau)\|_{L^2}^{\f 12} \|\na\p_3^2 u(\tau)\|_{L^{2}} ) \,d\tau\nonumber\\
  &\le C E_2^{\f 12}(t) E_0^{\f 12}(t)\Big(\int_{0}^{t} (1+t-\tau)^{-m}(1+\tau)^{-\f 1 4}d\tau
  + \int_{0}^{t} (1+t-\tau)^{-m}(1+\tau)^{-\f {3} 8+\f {\varepsilon}{2}}\, d\tau\Big)\nonumber\\
  &\le C  E(t)(1+t)^{-\f 1 4}.
  \end{align*}
 Thus, 
 \begin{align*}
  K_{3332}\le C   E(t)(1+t)^{-\f 1 4}.
  \end{align*}
which, together with (\ref{K3331}), yields
  \begin{align}\label{K333}
  K_{333}\le C   E(t)(1+t)^{-\f 1 4}.
  \end{align}
 As a consequence of (\ref{K331}) and (\ref{K333}), 
  \begin{align*}
	(1+t)^{\f 14}\|\p_3^2u(t)\|_{L^2}
	\le C E(t)+C(\|(\p_3^2u_0,\p_3^2b_0)\|_{L_{x_3}^2L_{x_1x_2}^1}+\|(\p_3^2u_0,\p_3^2b_0)\|_{L^2}).
  \end{align*}
 Combining all the estimates for the four cases above, we derive the desired estimate (\ref{E23}). This completes the proof of Lemma \ref{lem57}.
\end{proof}

\vskip .2in
Proposition \ref{prop3} then follows from the estimates (\ref{E20}), (\ref{E21}), (\ref{E22}) and (\ref{E23}). This completes the proof of Proposition \ref{prop3}.

\vskip .3in
\section*{\bf Acknowledgments}

Lin was partially supported by the National Natural Science Foundation of China (NNSFC) under Grant 11701049 and the China Postdoctoral Science Foundation under Grant 2017M622989.  Wu was partially supported by the National Science Foundation of the United States under grant DMS 2104682 and the AT\&T Foundation at Oklahoma State University. Zhu is partially supported by Shanghai Sailing Program under Grant 18YF1405500 and NNSFC under Grant 11801175.

\vskip .3in
\bibliographystyle{plain}

\end{document}